\documentclass[preprint,12pt]{elsarticle}
\usepackage{amsfonts,amsmath,amssymb}
\usepackage{indentfirst,setspace}
\usepackage{url,float}
\usepackage{color}
\usepackage[a4paper,ignoreall]{geometry}
\usepackage{longtable}
\usepackage{graphicx}
\usepackage{multicol}
\usepackage{stfloats}
\usepackage{enumerate}
\usepackage{amsthm}
\usepackage{multirow}
\usepackage{booktabs}

\usepackage[colorlinks]{hyperref}
\usepackage[justification=centering]{caption}

\newcommand{\F}{\mathbb {F}}
\def\qu#1 {\fbox {\footnote {\ }}\ \footnotetext { From Qu: {\color{red}#1}}}
\def\hqu#1 {}

\newtheorem{thm}{Theorem}[section]

\newtheorem{lem}[thm]{Lemma}
\newtheorem{rem}[thm]{Remark}

\newtheorem{pro}[thm]{Proposition}
\newtheorem{prob}[thm]{Problem}
\newtheorem{ex}[thm]{Example}

\begin{document}

\begin{frontmatter}



\title{Further Study of Planar Functions in Characteristic Two\tnoteref{label1}}
\tnotetext[label1]{This work is supported in part by the Nature Science Foundation of China (NSFC) under Grant 61722213, 11531002, in part by the National Key Research and Development Program of China under 2017YFB0802000, and in part by the Open Foundation of State Key Laboratory of Cryptology.}

\author[nudt]{Yubo Li}
\ead{leeub\_0425@hotmail.com}
\author[nudt]{Kangquan Li}
\ead{likangquan11@nudt.edu.cn}
%\author[nudt,SKL]{Longjiang Qu\textsuperscript{\Letter}}
\author[nudt,SKL]{Longjiang Qu\corref{cor1}}
\cortext[cor1]{Corresponding author.}
\ead{ljqu\_happy@hotmail.com}
\author[nudt]{Chao Li}
\ead{lichao\_nudt@sina.com}
\address[nudt]{College of Liberal Arts and Sciences, National University of Defense Technology, Changsha,	410073, China.}
\address[SKL]{State Key Laboratory of Cryptology, Beijing, 100878, China.}
%\title{}


%\author{}

%\address{}

\begin{abstract}
Planar functions are of great importance in the constructions of DES-like iterated ciphers, error-correcting codes, signal sets and the area of mathematics.
They are defined over finite fields of odd characteristic originally and generalized by Y. Zhou \cite{Zhou} in even characteristic. 
In 2016, L. Qu \cite{Q} proposed a new approach to constructing quadratic planar functions over $\F_{2^n}$.
Very recently, D. Bartoli and M. Timpanella \cite{Bartoli} characterized the condition on coefficients $a,b$ such that the function $f_{a,b}(x)=ax^{2^{2m}+1}+bx^{2^m+1} \in\F_{2^{3m}}[x]$ is a planar function over $\F_{2^{3m}}$ by the Hasse-Weil bound.

In this paper, using the Lang-Weil bound, a generalization of the Hasse-Weil bound, and the new approach introduced in \cite{Q}, we completely characterize the necessary and sufficient conditions on coefficients of four classes of planar functions over $\F_{q^k}$, where $q=2^m$ with $m$ sufficiently large (see Theorem \ref{main}).
The first and last classes of them are over $\F_{q^2}$ and $\F_{q^4}$ respectively, while the other two classes are over $\F_{q^3}$. 
One class over $\F_{q^3}$ is an extension of $f_{a,b}(x)$ investigated in \cite{Bartoli}, while our proofs seem to be much simpler.  
In addition, although the planar  binomial over $\F_{q^2}$ of our results is finally a known planar monomial, we also answer the necessity at the same time and solve partially an open problem for the binomial case proposed in  \cite{Q}. 
\end{abstract}

\begin{keyword} 
Planar Function, Permutation Polynomial, Hypersurface, Lang-Weil Bound

\MSC[2010] 94A60 11T06 11G20 12E05


\end{keyword}

\end{frontmatter}


\section{Introduction}
Let $p$ be an odd prime and $n$ be a positive integer. 
A function $f:\F_{p^n}\to\F_{p^n}$ is called \textit{planar} if the mapping $D_{f}$ 
\begin{equation*}
%\label{D_{f}}
x\mapsto f(x+a)-f(x)
\end{equation*}
is a permutation over $\F_{p^n}$ for each $a\in\F_{p^n}^{*}$. Planar functions in odd characteristic were first introduced by P. Dembowski and T.G. Ostrom \cite{D-O} in order to construct finite projective planes in 1968. 
Apart from this, they were also used in the constructions of DES-like iterated ciphers, error-correcting codes and signal sets.
For example, in 1975, K.J. Ganley and E. Spence \cite{MGES} showed that planar functions give rise to certain relative difference sets. 
Later in CRYPTO'92, K. Nyberg and L.R. Knudsen \cite{Nyberg} studied planar functions for applications in cryptography. 
In the cryptography literature, planar functions are called \textit{perfect nonlinear functions} (PN), since they are optimally resistant to differential cryptanalysis.
In the last decades, planar functions have been found to have deep applications in different areas of mathematics and codes.
Particularly in 2005, C. Carlet et al. \cite{CDJ}  utilized planar functions to construct error-correcting codes, which were then employed to design secret sharing schemes.
Later in 2007, C. Ding and J. Yin \cite{Ding-J} used planar functions to investigate signal sets and constructed optimal codebooks meeting the Levenstein bound.

However, if $p=2$, there are no planar functions under the conventional definition over $\F_{2^n}$. Since if $x$ satisfies $f(x+a)-f(x)=d$, where $d\in\F_{2^n}$, then so does $x+a$.
In \cite{Zhou}, an extended definition of planar functions in even characteristic was proposed by Y. Zhou.
That is, a function $f:\F_{2^n}\to\F_{2^n}$ is called \textit{planar} if the mapping
\begin{equation}
\label{pseudo-planar}
x\mapsto f(x+a)+f(x)+ax
\end{equation}
is a permutation over $\F_{2^n}$ for each $a\in\F_{2^n}^{*}$.

Note that such functions in even characteristic are still called  `planar'  by Y. Zhou \cite{Zhou}, while they are called `pseudo-planar'  in \cite{Abdukhalikov, Q}  to avoid confusion with planar functions in odd characteristic.
In the rest of the paper, we mainly investigate the functions over finite fields with even characteristic and still call them planar functions without ambiguity.

Moreover, these new planar functions over $\F_{2^n}$, as an analogue of planar functions in odd characteristic, also bring about finite projective planes and have the same types of applications as odd-characteristic planar functions do.
A planar function over $\F_{2^n}$ not only gives rise to a finite projective plane, a relative difference set and a presemifield, it also leads to a complete set of \textit{mutually unbiased base} (MUB) in $\mathbb{C}^{2^n}$ (where $\mathbb{C}$ is the Hilbert space), an optimal $(2 ^{2n}+2^n,2^n)$ complex codebook meeting the Levenstein bound, and a compressed sensing matrix with low coherence.
These interesting links are the motivations for the authors to study the constructions of planar functions over $\F_{2^n}$.
Substantial efforts have been directed toward such planar functions and their applications in even characteristic in the last several years.
For example, in 2014, K.-U. Schmidt and Y. Zhou \cite{Schmidt} showed that a planar function can be used to produce a finite projective plane, a relative difference set with parameters $(2^n,2^n,2^n,1)$, and certain codes with unusual properties. 
Meanwhile, Z. Zhou et al. \cite{ZDL} also pointed out that codebooks achieving the Levenstein bound can be used in compressed sensing, whose central problem is the construction of the compressed sensing matrix.
Compressed sensing is a novel sampling theory, which provides a fundamentally new approach to data acquisition.
Later in 2015, K. Abdukhalikov \cite{Abdukhalikov} used planar functions to give new explicit constructions of complete sets of MUBs, and showed a connection between quadratic planar functions and commutative presemifields.

A number of recent research works in the theory of coding, polynomials and algebraic geometry have been dedicated to the constructions of planar functions in characteristic two.
The simplest cases are planar monomials.
Three families of planar monomials were got by K.-U. Schmidt and Y. Zhou \cite{Schmidt} and Z. Scherr and M.E. Zieve \cite{Scherr} as follows.
\begin{enumerate}[(i)]
\item $f(x)=cx^{2^m}$, where $c\in\F_{2^n}^{*}$ (Trivial);
\item $f(x)=cx^{2^m+1}$ over $\F_{2^{2m}}$, where $c\in\F_{2^m}^{*}$ and ${\rm Tr}_{\F_{2^m}/\F_{2}}(c)=0$, where ${\rm Tr}_{\F_{2^m}/\F_{2}}(c)=c+c^2+\cdots+c^{2^{m-1}}$ is the absolutely trace function over $\F_{2^m}$ (see \cite[Theorem 3.1]{Schmidt}), and it was generalized by \cite[Theorem 26]{Q} that $c\in\F_{2^{2m}}^{*}$ and ${\rm Tr}_{\F_{2^m}/\F_{2}}(c^{2^m+1})=0$;
\item $f(x)=cx^{2^{2m}+2^m}$, where $c\in\F_{2^{3m}}^{*}$ and $m$ is even, furthermore, $c^{2^{2m}+2^m+1}=1$ and $c^{(2^{2m}+2^m+1)/3}\neq 1$ (see \cite[Theorem 1.1]{Scherr}).
\end{enumerate}

In addition, besides planar monomials, the next simplest cases are planar binomials.
In 2015, S. Hu, et al. \cite{HLZFG} introduced three families of planar binomials over $\F_{2^{3m}}$.

\begin{enumerate}[(i)]
\item $f(x)=a^{-(2^m+1)}x^{2^m+1}+a^{2^{2m}+1}x^{2^{2m}+1}\in\F_{2^{3m}}[x]$, where the parameter $a$ satisfies a trace equation (see \cite[Proposition 3.2]{HLZFG} for more details);
\item $f(x)=x^{2^{m}+1}+x^{2^{2m}+2^{m}}\in\F_{2^{3m}}[x]$, where $m\not\equiv2\pmod3$ (see \cite[Proposition 3.6]{HLZFG});
\item $f(x)=x^{2^{2m}+1}+x^{2^{2m}+2^{m}}\in\F_{2^{3m}}[x]$, where $m\not\equiv1\pmod3$ (see \cite[Proposition 3.8]{HLZFG}).
\end{enumerate}

Recently, D. Bartoli and K.-U. Schmidt \cite{Bartoli-K.-U} classified planar polynomials of degree at most $q^{1/4}$ on $\F_{q}$ and showed that such polynomials are precisely those in which the degree of every monomial is a power of two.
However, it is still open to classify the planar functions. 
Only the classification of the planar monomials was studied (see \cite{PMMEZ,Scherr,Schmidt}), and it was conjectured that there are only three families of such monomials \cite[Conjecture 3.2]{Schmidt} as above.

Furthermore, L. Qu \cite{Q} proposed a new approach to constructing quadratic planar functions.
According to (\ref{pseudo-planar}), a quadratic function $f$ over $\F_{2^n}$ is planar if and only if $$\mathbb{L}_{a}(x):=f(x+a)+f(x)+f(a)+ax$$
is a linearized permutation polynomial for each $a\in\F_{2^n}^{*}$.
L. Qu \cite{Q} converted it to studying the permutation property of dual polynomial $\mathbb{L}_{b}^{*}(a)$ (see \cite[Theorem 14]{Q} for more details), simplifying the conditions of the functions being planar.
Through this approach, L. Qu constructed several new explicit families of quadratic planar functions over $\F_{2^n}$ and revisited some known families.

Both S. Hu et al. \cite{HLZFG} and L. Qu \cite{Q} showed some concrete and separate planar functions over $\F_{2^{n}}$, however, they did not give the complete characterizations of the coefficients of these functions.
Very recently in 2020, D. Bartoli and M. Timpanella \cite{Bartoli} took into account the planar property of this type of function $$f_{a,b}(x)=ax^{2^{2m}+1}+bx^{2^m+1}\in\F_{2^{3m}}[x].$$
By using basic tools from algebraic geometry over finite fields, they completely determined the condition on $a,b\in\F_{2^{3m}}$ such that $f_{a,b}(x)$ is a planar function over $\F_{2^{3m}}$.

In this paper, by combining the approach of L. Qu \cite{Q} and the Lang-Weil bound (see Lemma \ref{Lang-Weil bound}, \ref{lang-weil}), we completely determine the coefficients of four classes of planar functions over $\F_{q^k}$ with $k=2,3,4$, where $q=2^m$ and $m$ is sufficiently large.
One class is a generalization of a known family and another class answers partially an open problem over $\F_{q^2}$ for the binomial case proposed by L. Qu \cite{Q}.
The method of this paper can be summarized as follows. By the approach of L. Qu, the planar property of function $f$ over $\F_{q^k}$ can be transformed into determining whether some equation has no solutions in $\F_{q^k}^{*}$, see Lemmas \ref{t=2}, \ref{t=3} and \ref{t=4}. 
Furthermore, the latter problem is to compute if the number of rational points on $\F_q$  of some hypersurface $\mathcal{C}$ over $\F_{q}$ associated with $f$ is $0$.
For the necessity of the planar property of $f$, the Lang-Weil bound tells us that if $\mathcal{C}$ is absolutely irreducible and $m$ is sufficiently large, then the number of rational points on $\F_q$ of $\mathcal{C}$ is larger than $0$ and thus $f$ is not planar. 
Hence the key problem here is to determine the condition such that $\mathcal{C}$ is not absolutely irreducible.  
As for the sufficiency, it can be proved by the definition of planar functions directly.

To sum up, we prove the following result.
\begin{thm}
\label{main}
Let $q=2^m$ with $m$ sufficiently large and $\F_{q^k}$ be an extension of $\F_{q}$. 
\begin{enumerate}[(1)]
\item When $k=2$, let $P_{1}(x)=ax^{q+1}+bx^{2(q+1)}\in\F_{q^2}[x]$.
Then $P_{1}(x)$ is a planar function if and only if
$$(a,b)=\Big(\frac{s^q}{1+s^{1+q}},0\Big),$$ where $s\in\F_{q^2}$ such that $1+s^{1+q}\neq0$.
\item When $k=3$, let $P_{2}(x)=ax^{q+1}+bx^{q^2+q}+cx^{q^2+1}\in\F_{q^3}[x]$.
Then $P_{2}(x)$ is a planar function if and only if $$(a,b,c)=\Bigg(\frac{v^q+u^{q+q^2}+u^{q^2}v^{1+q}}{1+\Delta},\frac{u^{q^2}v^{q}}{1+\Delta},\frac{v^{q+q^2}+u^{q^2}+u^{1+q^2}v^{q}}{1+\Delta}\Bigg),$$
where $u,v\in\F_{q^3}$ such that $\Delta=uv^q+u^{q}v^{q^2}+u^{q^2}v+u^{1+q+q^2}+v^{1+q+q^2}\neq1$.	
\item When $k=3$, let $P_{3}(x)=ax^{2(q+1)}+bx^{2(q^2+q)}+cx^{2(q^2+1)}\in\F_{q^3}[x]$.
Then $P_{3}(x)$ is a planar function if and only if $b=0$ and $c=a^q$.
\item When $k=4$, let $P_{4}(x)=ax^{q+1}+bx^{q^2+1}+cx^{q^3+1}\in\F_{q^4}[x]$.
Then $P_{4}(x)$ is a planar function if and only if
$$(a,b,c)=\Big(0,\frac{s_{1}^{q^2}}{1+s_{1}^{1+q^2}},0\Big)$$
or $$(a,b,c)=\Bigg(\frac{s_{2}^{q+q^2+q^3}}{1+s_{2}^{1+q+q^2+q^3}},\frac{s_{2}^{q^2+q^3}}{1+s_{2}^{1+q+q^2+q^3}},\frac{s_{2}^{q^3}}{1+s_{2}^{1+q+q^2+q^3}}\Bigg),$$
where $s_{1}\in\F_{q^4}$ such that $1+s_{1}^{1+q^2}\neq0$ and $s_{2}\in\F_{q^4}$ such that $1+s_{2}^{1+q+q^2+q^3}\neq0$.	
\end{enumerate}
\end{thm}

The paper is organized as follows. In Section 2, we recall some definitions and propose some useful lemmas. In Section 3, we give the complete characterization of the planar functions of the form $ P_{1}(x)=ax^{q+1}+bx^{2(q+1)}$ over $\F_{q^2}$. Although this class of planar function has been obtained by L. Qu \cite[Theroem 26]{Q}, we give the necessity at the same time. 
Moreover, in Section 4, we consider two classes of planar functions over $\F_{q^3}$ with the types $P_{2}(x)=ax^{q+1}+bx^{q^2+q}+cx^{q^2+1}$ and $P_{3}(x)=ax^{2(q+1)}+bx^{2(q^2+q)}+cx^{2(q^2+1)}$.
The former generalizes the main result of D. Bartoli and M. Timpanella \cite[Theorem 2.6]{Bartoli}, and the latter is an extension work of \cite[Theorem 16]{Q}.
Meanwhile, we prove that the condition (i.e. $b=0$ and $c=a^q$) is also necessary for $P_{3}(x)$ being planar.  
Section 5 involves the planar functions over $\F_{q^4}$ with the form $P_{4}(x)=ax^{q+1}+bx^{q^2+1}+cx^{q^3+1}$.
In Section 6, we discuss the equivalence between the semifields produced by the planar functions in Theorem \ref{main} and the corresponding fields.
Finally, Section 7 is the conclusion.

\section{Preliminaries}
In this section, we give necessary definitions and results which will be frequently used in this paper.

\subsection{Presemifield, Semifields and their Equivalence}
A presemifield is a ring with no zero-divisor, and with left and right distributivity \cite{Dembowski}.
A presemifield with multiplicative identity is called a semifield.
A finite presemifield can be obtained from a finite field $(\F_{2^n},+,\cdot)$ by introducing a new product operation , so it is denoted by $(\F_{2^n},+,*)$.
An isotopism between two presemifields $\mathbb{S}_{1}=(\F_{2^n},+,*)$ and $\mathbb{S}_{2}=(\F_{2^n},+,\star)$ is a triple $(M,N,L)$ of bijective linearized mapping $\F_{2^n}\mapsto\F_{2^n}$ such that
$$M(x)*N(y) = L(x\star y), \textup{~for all~}x, y \in\F_{2^n}.$$
Furthermore, if $M = N$, then $\mathbb{S}_{1}$ is strongly isotopic to $\mathbb{S}_{2}$.
In particular, for even characteristic, there is a result on the commutative (pre)semifields obtained by R.S. Coulter and M. Henderson \cite{Coulter-Henderson}.
\begin{lem}
\cite[Corollary 2.7]{Coulter-Henderson}
Two commutative presemifields of even order are isotopic if and only if they are strongly isotopic.
\end{lem}
The isotopism is the most important equivalence relation between (pre)semifield, since A.A. Albert \cite{Albert} showed that two (pre)semifields coordinate isomorphic planes if and only if they are isotopic. 
By isotopism we can also get a semifield $\mathbb{S}$ from a presemifield $\mathbb{P}$. 
Let $*$ be the multiplication of a presemifield. Then for every $0\neq e\in\mathbb{P}$ we obtain a semifield multiplication $\star$ defined by:
$$(x*e)\star(y*e)=x*y,$$
with the identity $e*e$.
Moreover, if $(\F_{2^n},+,*)$ is commutative then so is each such semifield $(\F_{2^n},+,\star)$.

The study of finite commutative semifields was begun by Dickson.
Since then the only examples found have been Knuth's binary semifields \cite{Knuth}.
Moreover, to the best of the author's knowledge, there are only two types of presemifields with even characteristic, that is, finite fields and the Kantor family of commutative presemifields \cite{Kantor}.
Assume that we have a chain of fields $\F=\F_{0}\supset\F_{1}\supset\cdots\supset\F_{n}$ of characteristic 2 with $[\F:\F_{n}]$ odd and corresponding trace mappings $\textup{Tr}_{i}: \F\mapsto \F_{i}$.
In 2003, W.M. Kantor \cite{Kantor} presented commutative presemifields $\mathbf{B}((\F_{i})^{n}_{0},(\zeta_{i})^{n}_{1})$ on which the multiplication is defined as:
$$x*y=xy+\Big(x\sum_{i=1}^{n}\textup{Tr}_{i}(\zeta_{i}y)+y\sum_{i=1}^{n}\textup{Tr}_{i}(\zeta_{i}x)\Big)^2,$$
where $\zeta_{i}\in\F^{*},1\le i\le n$. These semifields are related to a subfamily of the symplectic spreads constructed in \cite{Kantor-Williams}.
Note that this semifield is a generalization of Knuth's binary semifields \cite{Knuth}, on which the multiplication is defined as:
$$x*y=xy+(x\textup{Tr}(y)+y\textup{Tr}(x))^2,$$
corresponding to the presemifields $\mathbf{B}((\F_{i})^{1}_{0},(1))$. The planar function derived from Knuth's semifield is $(x\textup{Tr}(x))^2$.

It is well known that commutative semifields (up to isotopism) over finite fields of characteristic two correspond to quadratic planar functions (see \cite[Theorem 9]{Abdukhalikov} for instance). 
In particular, if $P$ is a quadratic planar function over $\F_{2^n}$, then $(\F_{2^n},+,*)$ with multiplication $x*y=xy+P(x+y)+P(x)+P(y)$ is a presemifield.
Conversely, if $(\F_{2^n},+,\star)$ is a commutative presemifield, then there also exist a strongly isotopic commutative presemifield $(\F_{2^n},+,*)$ and a planar function $P$ such that $x*y=xy+P(x+y)+P(x)+P(y)$. 
Equivalence between quadratic planar functions is the same as isotopism between the corresponding (pre)semifields (see \cite[Proposition 3.4]{Zhou} for detail).

Let $\mathbb{S}=(\F_{2^n},+,*)$ be a semifield. The subsets
\begin{eqnarray*}
N_{l}(\mathbb{S}) & = & \{\alpha\in\mathbb{S}~|~(\alpha*x)*y=\alpha*(x*y) \textup{~for all~} x,y\in\mathbb{S}\}, \\
N_{m}(\mathbb{S}) & = & \{\alpha\in\mathbb{S}~|~(x*\alpha)*y=x*(\alpha*y) \textup{~for all~} x,y\in\mathbb{S}\}, \\
N_{r}(\mathbb{S}) & = & \{\alpha\in\mathbb{S}~|~(x*y)*\alpha=x*(y*\alpha) \textup{~for all~} x,y\in\mathbb{S}\}.
\end{eqnarray*}
are called the left, middle and right nucleus of $\mathbb{S}$, respectively.
A recent survey about finite semifields can be found in \cite{Lavrauw-Polverino}.
Furthermore, if the planar functions are of Dembowski-Ostrom type, then the equivalence on them is the same as the isotopism of the corresponding semifields.
To check whether a semifield is new or not, a natural way is to determine its left (right) nucleus.

\subsection{Algebraic hypersurfaces and Lang-Weil Bound}
Let $\overline{\F}_{q}$ be the algebraic closure of $\F_{q}$ and $\#M$ be the cardinality of set $M$.
A polynomial $f\in\F_{q}[X_{0},X_{1},\dots,X_{k}]$ is said to be absolutely irreducible if it is irreducible in $\overline{\F}_{q}[X_{0},X_{1},\dots,X_{k}]$. If $f\in\F_{q}[X_{0},X_{1},\dots,X_{k}]$ is homogeneous, define $$V_{\mathbb{P}^{k}(\F_{q})}(f)=\big\{(x_{0}:\cdots:x_{k})\in\mathbb{P}^{k}(\F_{q}):f(x_{0},\dots,x_{k})=0\big\}.$$
In the sequel, we list a well-known result called the Lang-Weil bound (see \cite{ACGM,SLAW} for more details), aiming to estimate the rational points of an absolutely irreducible $\F_{q}$-hypersurface of degree $d$. 
\begin{lem}
\label{Lang-Weil bound}
\cite[Theorem 1]{SLAW}
(Lang-Weil bound) Let $f\in\F_{q}[X_{0},X_{1},\dots,X_{k}]$ be an absolutely irreducible homogeneous polynomial of degree $d$. Then
$$\Big|\#V_{\mathbb{P}^{k}(\F_{q})}(f)-q^{k-1}\Big|\le(d-1)(d-2)q^{k-\frac{3}{2}}+ c(n,d)q^{k-2},$$
where $c(n,d)$ is a constant depending only on $n$ and $d$. 
\end{lem}
Note that when $k = 2$, Lemma \ref{Lang-Weil bound} is actually the so-called Hasse-Weil bound (see \cite{HXD,Stichtenoth} for more details).
Moreover, A. Cafure and G. Matera \cite{ACGM} provided an explicit expression for the constant $c(n,d)$.
\begin{lem}
\label{lang-weil}
\cite[Theorem 5.2]{ACGM}
Let $f\in\F_{q}[X_{0},X_{1},\dots,X_{k}]$ be an absolutely irreducible homogeneous polynomial of degree $d$. Then
$$\Big|\#V_{\mathbb{P}^{k}(\F_{q})}(f)-q^{k-1}\Big|\le(d-1)(d-2)q^{k-\frac{3}{2}}+ 5\cdot d^{\frac{13}{3}}q^{k-2}.$$
\end{lem}

\subsection{Other Results}
In this subsection, we firstly review some necessary definitions and results for future use.
Throughout this paper, we always denote $2^m$ by $q$ and for any element $x\in\F_{q^k}$, $${\rm Tr}_{\F_{q^k}/\F_{q}}(x):=x+x^q+\cdots + x^{q^{k-1}}$$ is the relatively trace function from $\F_{q^k}$ to $\F_{q}$, where $\F_{q^k}$ is an extension of $\F_{q}$ with dimension $k$.
Particularly, when $q=2$, ${\rm Tr}_{\F_{2^k}/\F_{2}}$ is the absolutely trace function over $\F_{2^k}$. 

\begin{lem}
\label{Dickson}
\cite[Page 362]{Lidl}
Let $\F_{q^k}$ be an extension of $\F_{q}$. Then the linearized polynomial $$L(x)=\sum_{i=0}^{k-1}a_{i}x^{q^i}\in\F_{q^k}[x]$$ is a permutation polynomial of $\F_{q^k}$ if and only if the Dickson determinant of $a_{0},a_{1},\cdots,a_{k-1}$ is nonzero, that is,
\begin{displaymath}
\det
\left( \begin{array}{ccccc}
a_{0} & a_{1} & a_{2} & \ldots & a_{k-1} \\
a_{k-1}^q & a_{0}^q & a_{1}^q & \ldots & a_{k-2}^q \\
a_{k-2}^{q^2} & a_{k-1}^{q^2} & a_{0}^{q^2} & \ldots & a_{k-3}^{q^2}\\
\vdots & \vdots & \vdots & & \vdots\\
a_{1}^{q^{k-1}} & a_{2}^{q^{k-1}} & a_{3}^{q^{k-1}} & \ldots & a_{0}^{q^{k-1}} 
\end{array} \right)\neq 0.
\end{displaymath}
\end{lem}

Moreover, the following three lemmas, i.e., Lemmas \ref{t=2}, \ref{t=3} and \ref{t=4},  transform the problem of proving the planar property of a family of quadratic polynomials into that of showing some equation has no solutions in $\F_{q^2}^{*}$, $\F_{q^3}^{*}$ and $\F_{q^4}^{*}$ respectively.

\begin{lem}
\label{t=2}
\cite[Theorem 25]{Q}
Let $q=2^m$ and $$F(x)=\sum_{i=0}^{m-1}c_{i}x^{2^{m+i}+2^i}\in\F_{q^2}[x].$$
Then $F$ is planar over $\F_{q^2}$ if and only if $$x^{2^m+1}+\sum_{i=0}^{m-1}(c_{i}x)^{2^{m-i+1}}+\sum_{i=0}^{m-1}(c_{i}x)^{2^{2m-i+1}}=0$$ has no solutions in $\F_{q^2}^{*}$.
\end{lem}

\begin{lem}
\label{t=3}
\cite[Theorem 15]{Q}
Set $q=2^m$ and $$F(x)=\sum_{i=0}^{2m-1}c_{1,i}x^{2^{m+i}+2^i}+\sum_{i=0}^{m-1}c_{2,i}x^{2^{2m+i}+2^i}\in\F_{q^3}[x].$$
Then $F$ is planar over $\F_{q^3}$ if and only if $$x^{q^2+q+1}+{\rm Tr}_{\F_{q^3}/\F_{q}}(x^qA_{2}^2)=0$$ has no solutions in $\F_{q^3}^{*}$, where
$$A_{2}=\sum_{i=0}^{m-1}(c_{2,i}x)^{2^{3m-i}}+\sum_{i=0}^{2m-1}(c_{1,i}x)^{2^{2m-i}}.$$
\end{lem}

\begin{lem}
\label{t=4}
\cite[Theorem 22]{Q}
Assume $q=2^m$ and $$F(x)=\sum_{i=0}^{3m-1}c_{1,i}x^{2^i(q+1)}+\sum_{i=0}^{2m-1}c_{2,i}x^{2^i(q^2+1)}+\sum_{i=0}^{m-1}c_{3,i}x^{2^i(q^3+1)}\in\F_{q^4}[x].$$
Then $F$ is planar over $\F_{q^4}$ if and only if $$x^{q^3+q^2+q+1}+A_{2}^{2q+2}+(A_{3}^{2q^2+2}+A_{3}^{2q^3+2q})+(x^{q^2+1}A_{2}^{2q}+x^{q^3+q}A_{2}^2)+{\rm Tr}_{\F_{q^4}/\F_{q}}(x^{q^2+q}A_{3}^2)=0$$ has no solutions in $\F_{q^4}^{*}$, where
\begin{displaymath}
\left\{ \begin{array}{l}
\displaystyle
A_{2}=\sum_{i=0}^{2m-1}\Big((c_{2,i}x)^{2^{4m-i}}+(c_{2,i}x)^{2^{2m-i}}\Big), \\
\displaystyle
A_{3}=\sum_{i=0}^{m-1}(c_{3,i}x)^{2^{4m-i}}+\sum_{i=0}^{3m-1}(c_{1,i}x)^{2^{3m-i}}.
\end{array} \right.
\end{displaymath}
\end{lem}

Then, as stated before, the problem of showing the planar property can be transformed into that of proving some relative equation has no nonzero solutions. 
Next we recall the usage of the normal basis that dedicated to this problem.
It is well known that there exists some bijection between $\F_{q^k}$ and $\F_q^k$. For instance, 	let $\{\xi,\xi^q,\cdots,\xi^{q^{k-1}}\}$ be a normal basis of $\F_{q^k}$ over $\F_{q}$. Then for any element $\varepsilon\in\F_{q^k}$, there exists a unique element $(\varepsilon_{0}, \varepsilon_{1}, \cdots, \varepsilon_{k-1})\in \F_q^k$ such that $\varepsilon=\varepsilon_{0}\xi+\varepsilon_{1}\xi^q+\cdots+\varepsilon_{k-1}\xi^{q^{k-1}},$ and vice versa. For convenience, we denote the bijection stated above by $\varPhi: \F_{q}^k\mapsto \F_{q^k}$ and its inverse is denoted by $\varPhi^{-1}$.   
Let  $h(x)\in\F_{q^k}[x]$  and $g(x)$ be a mapping from $\F_{q^k}$ to $\F_{q}$ defined by
\begin{equation}
\label{g}g(x)={\rm Tr}_{\F_{q^k}/\F_{q}}\big(h(x)\big)+x^{1+q+q^2+\cdots+q^{k-1}}.
\end{equation} 	
Clearly, there exists some polynomial over $\F_{q^k}$ denoted by $G(X_{0},X_{1},\cdots,X_{k-1})$ such that for any $\varepsilon\in\F_{q^k}$, $$ G(\varepsilon,\varepsilon^q,\cdots,\varepsilon^{q^{k-1}}) = g(\varepsilon). $$ We assume that
\begin{equation}
\label{G}G(X_{0},X_{1},\cdots,X_{k-1})=\varphi(X_{0},X_{1},\cdots,X_{k-1})+X_{0}X_{1} \cdots X_{k-1},
\end{equation} where $\varphi(\varepsilon,\varepsilon^q,\cdots,\varepsilon^{q^{k-1}}) = {\rm Tr}_{\F_{q^k}/\F_{q}}\big(h(\varepsilon)\big)$ for any $\varepsilon\in\F_{q^k}$. Note that the coefficients of $G$ are in $\F_{q^k}$.  For example, if $k=2$ and $h(x)=ax^2+bx\in\F_{q^2}[x]$, then we can assume that $\varphi(X_0,X_1)=aX_0^2+bX_0+a^qX_1^2+b^qX_1$. 

\begin{lem}
\label{nonzero}
Let $k$ be a given positive integer, $q=2^m$, $g$ and $G$ be defined as in \eqref{g} and \eqref{G}, respectively. If $G$ is absolutely irreducible and $m$ is sufficiently large, then the equation $g(x)=0$ has at least one nonzero solution in $\F_{q^k}$.
\end{lem}

\begin{proof}
Let $\{\xi,\xi^q,\cdots,\xi^{q^{k-1}}\}$ be a normal basis of $\F_{q^k}$ over $\F_{q}$. We firstly define some polynomial $\Psi$ from $\F_q^k$ to $\F_q$ which is relative to $g$ and $G$. 	
Denote $Y_{0} = \varPhi(X_0,X_1,\cdots,X_{k-1}) = \xi X_{0}+\xi^q X_{1}+\cdots+\xi^{q^{k-1}} X_{k-1}$ from $\F_q^k$ to $\F_{q^k}$ and $Y_{j}=Y_{0}^{q^{j}}$ for $1\le j\le k-1$. Then we define
\begin{eqnarray*}
\label{Psi}
&&\Psi(X_{0},X_{1},\cdots,X_{k-1})= G(Y_{0},Y_{1},\cdots,Y_{k-1}) \nonumber\\
&=&G(\xi X_{0}+\xi^q X_{1}+\cdots+\xi^{q^{k-1}} X_{k-1},\cdots,\xi X_{1}+\xi^q X_{2}+\cdots+\xi^{q^{k-1}} X_{0}).
\end{eqnarray*}
Clearly, $\Psi$ is a mapping from $\F_q^k$ to $\F_q$. 
Then we will show that the polynomial $\Psi(X_{0}, X_{1},\cdots, X_{k-1})$ is in fact defined over $\F_{q}$.
Its terms can be divided into two parts.
	
The first part is $\varphi(Y_{0},Y_{1},\cdots,Y_{k-1})$. 
According to our definition, we have $\varphi( Y_{0},Y_{1},\cdots,Y_{k-1}) = {\rm Tr}_{\F_{q^k}/\F_{q}}\big(h(\xi X_{0}+\xi^q X_{1}+\cdots+\xi^{q^{k-1}} X_{k-1})\big)$.   
All the coefficients in the expanding expressions are of the form ${\rm Tr}_{\F_{q^k}/\F_{q}}(\cdot)$, which belong to $\F_{q}$ directly.
	
The second part is $Y_{0}Y_{1} \cdots Y_{k-1}$. It is clear that 
$$(Y_{0}Y_{1} \cdots Y_{k-1})^q=Y_{0}^{(1+q+\cdots+q^{k-1})\cdot q}=Y_{0}^{1+q+\cdots+q^{k-1}}=Y_{0}Y_{1} \cdots Y_{k-1}.$$
Then it follows that  $Y_{1}Y_{2} \cdots Y_{k}$ is in fact a polynomial over $\F_q$. 
	
Therefore, the polynomial $\Psi$ is defined on $\F_{q}$ and also absolutely irreducible due to the property of $G$. Moreover, let $\Psi^{'}(X_0,X_1,\cdots,X_{k-1},X_k)$ be the homogenization of $\Psi(X_{0},X_{1},\cdots,X_{k-1})$. Assume that the degree of $\Psi$ is $d$, so is that of $\Psi^{'}$. Let $\mathcal{C}$ be the hypersurface defined by $\Psi^{'}$.  Then according to Lemma \ref{lang-weil}, the number of rational points over $\F_q$ in $\mathcal{C}$  is at least
$$ q^{k-1}-(d-1)(d-2)q^{k-\frac{3}{2}}-5\cdot d^{\frac{13}{3}}q^{k-2}.$$ Moreover, let $X_k=0$ in $\Psi^{'}$. Then we can see that the number of rational points over $\F_q$ in $\mathcal{C}$  at infinity is at most $d\cdot q^{k-2}$. Therefore, if $m$ is sufficiently large, $\mathcal{C}$ has rational points over $\F_q$.
Thus the equation $g(x)=0$ has at least one nonzero solution in $\F_{q^k}$.
\end{proof}

In the following three sections, we consider the planar property of four classes of explicit functions over $\F_{q^k}$ with $k=2,3,4$ respectively.

\section{Planar functions over $\F_{q^2}$}
In this section, we mainly investigate one class of binomial planar functions over $\F_{q^2}$ and give the proof of (1) of Theorem \ref{main}.\\
%\begin{thm}
%\label{Th-2(q+1)-m>2}
%Let $q=2^m$ with $m$ sufficiently large and $P_{1}(x)=ax^{q+1}+bx^{2(q+1)}\in\F_{q^2}[x]$.
%Then $P_{1}(x)$ is a planar function if and only if
%$$(a,b)=\Big(\frac{s^q}{1+s^{1+q}},0\Big),$$ where $s\in\F_{q^2}$ such that $1+s^{1+q}\neq0$.	
%\end{thm}
%\begin{proof}
\textit{Proof of (1) of Theorem \ref{main}.}
According to Lemma \ref{t=2}, we know that $P_{1}(x)$ is a planar function if and only if $g(x)=0$ has no solutions in $\F_{q^2}^{*}$, where 
$$g(x) = x^{q+1}+(ax)^{2q}+(bx)^q+(ax)^2+bx.$$
Let
\begin{equation}
\label{curve-t=2}
G(X,Y)=XY+a^2X^2+a^{2q}Y^2+bX+b^qY.
\end{equation}
Then $G(\varepsilon,\varepsilon^q)\neq0$ for any $\varepsilon\in\F_{q^2}^{*}$ is exactly equivalent to that $g(x)=0$ has no solutions in $\F_{q^2}^{*}$, i.e., $P_1(x)$ is a planar function.

If $G(X,Y)$ is absolutely irreducible, we know that $P_{1}(x)$ is not planar according to Lemma \ref{nonzero}. 
Hence it suffices to consider only when $G(X,Y)$ is not absolutely irreducible.

We now determine the conditions of $a,b$ such that $\mathcal{C}_{a,b}$ defined by $G(X,Y)$ is not absolutely irreducible.
If so, it is clear that $G(X,Y)$ can only be factorized as $G(X,Y)=G_{1}G_{2}$, where $\deg(G_{1})=\deg(G_{2})=1$.
Moreover, at most one constant term of $G_{1}$ and $G_{2}$ is nonzero.
	
\textbf{Case 1:} One constant term of $G_{1}$ and $G_{2}$ is nonzero.
W.l.o.g., we assume that the constant term of $G_{1}$ is nonzero.
Obviously, $G_{1}^q$ does not coincide with $G_{2}$.
So we have $G_{1}^q$ and $G_{1}$ must coincide and the same with $G_{2}^q$ and $G_{2}$.
This indicates that $G_{2}$ must contain $X$ and $Y$ simultaneously, let $G_{2}=X+\alpha Y$, where $\alpha\in\overline{\F}_{q}^{*}$.
The two lines $X+\alpha Y=0$ and $Y+\alpha^q X=0$ coincide if and only if
\begin{displaymath}
\alpha^{q+1}=1.
\end{displaymath}
Hence, $\alpha\in\F_{q^2}^{*}$. In this case, the determinant of the matrix
\begin{displaymath}
M=\left( \begin{array}{cc}
1 & \alpha  \\
\alpha^q & 1  
\end{array} \right)
\end{displaymath}
vanishes. By Lemma \ref{Dickson}, there exists some $\varepsilon\in \F_{q^2}^{*}$ such that  $\varepsilon+\alpha \varepsilon^q = 0$. Then $$G(\varepsilon,\varepsilon^q) = G_1(\varepsilon,\varepsilon^q) G_2(\varepsilon,\varepsilon^q) = G_1(\varepsilon,\varepsilon^q) (\varepsilon+\alpha \varepsilon^q) = 0$$ and 
thus $P_{1}(x)$ is not planar in this case. 
	
\textbf{Case 2:} Both the constant terms of $G_{1}$ and $G_{2}$ are zero. From the proof of Case 1,  $P_{1}(x)$ is not planar if $G_{i}^q$ and $G_{i}$ coincide for $i=1 $ or $2$. Next we consider the case $G_{1}=G_{2}^q$, i.e.,
\begin{equation}
\label{X,Y-2}
G(X,Y)=G_{2}G_{2}^q=G_{2}^{1+q}.
\end{equation}
Since there are terms $X^2$ and $XY$ in the expression of $G(X,Y)$, $G_{2}$ must contain $X$ and $Y$ simultaneously.
Let $G_{2}=X+\beta Y$, where $\beta^{1+q}\neq1$.
Expanding (\ref{X,Y-2}), we have
\begin{equation}
\label{Expanding-X,Y-2}
G(X,Y)=(1+\beta^{1+q})XY+\beta Y^2+\beta^qX^2.
\end{equation}
Comparing the coefficients of (\ref{curve-t=2}) and (\ref{Expanding-X,Y-2}), we find that
$$(a^2,b)=\Bigg(\frac{\beta^{q}}{1+\beta^{1+q}},0\Bigg),$$
where $\beta^{1+q}\neq1$.
Let $\beta=s^2$.
Then
$$(a,b)=\Bigg(\frac{s^{q}}{1+s^{1+q}},0\Bigg),$$
where $s^{1+q}\neq1$.
	
It follows from (\ref{X,Y-2}) that  for any $\varepsilon\in\F_{q^2}^{*}$, $$G(\varepsilon,\varepsilon^q)=(\varepsilon+\beta\varepsilon^{q})^{1+q}.$$ 
The matrix
\begin{displaymath}
M=\left( \begin{array}{cc}
1 & \beta  \\
\beta^q & 1 
\end{array} \right)
\end{displaymath}
has determinant $1+\beta^{1+q}$, which does not equal to 0. From Lemma \ref{Dickson}, there does not exist $\varepsilon\in\F_{q^2}^{*}$ such that $\varepsilon+\beta \varepsilon^q =0$,
that is to say, $G(\varepsilon,\varepsilon^q)\neq0$ for any $\varepsilon\in\F_{q^2}^{*}$. Hence $P_{1}(x)$ is planar over $\F_{q^2}$.
	
Thus the proof is completed.\hfill$\square$
%\end{proof}

\begin{rem}
\emph{Note that the planar function obtained in (1) of Theorem \ref{main} is actually a monomial, which has been presented by L. Qu \cite[Theorem 26]{Q} as follows: if $c\in\F_{q^2}^{*}$ and 	${\rm Tr}_{\F_{q}/\F_{2}}(c^{1+q})=0$, then $f(x)=cx^{q+1}$ is a planar function. In the following, we show that the condition of $c$ as above is indeed equivalent to that of $a$ in (1) of Theorem \ref{main}. }
\end{rem}

\begin{pro}
\label{M=N} 	
Let $q=2^m$, $M=\displaystyle\Big\{c\in\F_{q^2}~\Big|~{\rm Tr}_{\F_{q}/\F_{2}}(c^{1+q})=0\Big\}$ and $N=\displaystyle\Big\{a\in\F_{q^2}~\Big|~a=\frac{s^{q}}{1+s^{1+q}},1+s^{1+q}\neq0,s\in\F_{q^2}\Big\}$. Then $M=N$.
\end{pro}

\begin{proof}
First, for any $a\in N$, we have $$a^{1+q}=\big(\frac{s^{q}}{1+s^{1+q}}\big)^{1+q}=\frac{s^{1+q}}{(1+s^{1+q})^2}=\frac{1}{(1+s^{1+q})^2}+\frac{1}{1+s^{1+q}}.$$ 
Meanwhile, we have $s^{1+q}\in\F_{q}$, since $(s^{1+q})^q=s^{1+q}$.
Based on this, 
$${\rm Tr}_{\F_{q}/\F_{2}}(a^{1+q})={\rm Tr}_{\F_{q}/\F_{2}}\Big(\frac{1}{(1+s^{1+q})^2}+\frac{1}{1+s^{1+q}}\Big)=0.$$
Thus $a\in M$.
Therefore, $N\subseteq M$.
	
Next, the number of $c^{1+q}$ satisfying the condition ${\rm Tr}_{\F_{q}/\F_{2}}(c^{1+q})=0$ is $2^{m-1}=\displaystyle\frac{q}{2}$.
Moreover, for any $b\in\F_{q}^{*}$, the equation $c^{1+q}=b$ has $1+q$ solutions in $\F_{q^2}$ indeed. 
Hence, $\#M=\displaystyle(\frac{q}{2}-1)(q+1)+1=\frac{q^2-q}{2}$.
	
Then, we will show that $\#N=\displaystyle\frac{q^2-q}{2}$.
Before this, we establish that the mapping $$x\mapsto\frac{x^q}{1+x^{1+q}}$$ is 2-to-1 over $\F_{q^2}^{*}\backslash\mu_{q+1}$, where $\mu_{q+1}:=\{\delta\in\F_{q^2}\mid \delta^{1+q}=1\}$. It equals to prove that for any $s,t\in\F_{q^2}^{*}\backslash\mu_{q+1}$, the equation
\begin{equation}
\label{s,t neq 0}
\frac{s^q}{1+s^{1+q}}=\frac{t^q}{1+t^{1+q}}
\end{equation}only has two solutions, where one of them is $s=t$.
So we let $s\neq t$ in the sequel.
	
From Eq. (\ref{s,t neq 0}), we have $(s+t)^q=(s+t)(st)^q$. 
Due to the fact that $s+t\neq0$, we get
\begin{equation}
\label{s+t q-1}
(s+t)^{q-1}=(st)^q.
\end{equation}
Raising Eq. (\ref{s+t q-1}) into $(q+1)$-th power, we have
\begin{equation}
\label{q+1=1}
(st)^{1+q}=1.
\end{equation}
Combing Eq. (\ref{s,t neq 0}) and Eq. (\ref{q+1=1}), we obtain $$\displaystyle\frac{s^q}{1+s^{1+q}}=\frac{t^q}{1+t^{1+q}}=\frac{t^q}{(st)^{1+q}+t^{1+q}}=\frac{\frac{1}{t}}{1+s^{1+q}},$$
which means $t=\displaystyle\frac{1}{s^q}$.
Plugging $t=\displaystyle\frac{1}{s^q}$ into Eq. (\ref{s,t neq 0}), we find exactly that $$\displaystyle\frac{t^q}{1+t^{1+q}}=\frac{(\frac{1}{s^q})^q}{1+(\frac{1}{s^q})^{1+q}}=\frac{s^q}{1+s^{1+q}}.$$
Therefore, Eq. (\ref{s,t neq 0}) only has two solutions $s=t$ and $t=\displaystyle\frac{1}{s^q}$.
Moreover, $s=t$ and $t=\displaystyle\frac{1}{s^q}$ cannot coincide, since $s^{1+q}\neq1$.
	
Furthermore, the number of $s\in\F_{q^2}^{*}$ satisfying $s^{1+q}\neq1$ is $(q-2)(q+1)$.
Hence, $\#N=\displaystyle\frac{(q-2)(q+1)}{2}+1=\frac{q^2-q}{2}$.
	
That is to say, $\#M=\#N$.
Thus we have $M=N$. The proof is finished.
\end{proof}

Also note that in \cite{Q}, L. Qu proposed the following open problem.
\begin{prob}
\label{problem}
\cite[Problem 27]{Q}	
Set $q=2^m$. Let $$F(x)=\sum_{i=0}^{m-1}c_{i}x^{2^{m+i}+2^i}\in\F_{q^2}[x].$$
Is it true that $F$ is planar over $\F_{q^2}$ if and only if ${\rm Tr}_{\F_{2^m}/\F_{2}}(c_{0}^{q+1})=0$ and $c_{1}=c_{2}=\cdots=c_{m-1}=0$; or to find a counter-example?
\end{prob}
Clearly, (1) of Theorem \ref{main} answers partially the above problem for the binomial case. 

\section{Planar functions over $\F_{q^3}$}
In this section, we completely determine the asymptotic conditions (which means that m is sufficiently large) on $a,b,c\in\F_{q^3}$ such that two classes of functions over $\F_{q^3}$ with the forms $P_{2}(x)=ax^{q+1}+bx^{q^2+q}+cx^{q^2+1}$ and $P_{3}(x)=ax^{2(q+1)}+bx^{2(q^2+q)}+cx^{2(q^2+1)}$ are planar functions respectively.

%\begin{thm}
%\label{ax^{q+1}}
%Let $q=2^m$ with $m$ sufficiently large and $P_{2}(x)=ax^{q+1}+bx^{q^2+q}+cx^{q^2+1}\in\F_{q^3}[x]$.
%Then $P_{2}(x)$ is a planar function if and only if $$(a,b,c)=\Bigg(\frac{v^q+u^{q+q^2}+u^{q^2}v^{1+q}}{1+\Delta},\frac{u^{q^2}v^{q}}{1+\Delta},\frac{v^{q+q^2}+u^{q^2}+u^{1+q^2}v^{q}}{1+\Delta}\Bigg),$$
%where $u,v\in\F_{q^3}$ such that $\Delta=uv^q+u^{q}v^{q^2}+u^{q^2}v+u^{1+q+q^2}+v^{1+q+q^2}\neq1$.
%\end{thm}
First, we give the proof of (2) of Theorem \ref{main}.\\
%\begin{proof}
\textit{Proof of (2) of Theorem \ref{main}.}
According to Lemma \ref{t=3}, we know that $P_{2}(x)$ is a planar function if and only if $g(x)=0$ has no solutions in $\F_{q^3}^{*}$, where
$$g(x)=x^{q^2+q+1}+{\rm Tr}_{\F_{q^3}/\F_{q}}\Big(x^q\Big(cx+(ax)^{q^2}+(bx)^q\Big)^2\Big).$$
Let
\begin{eqnarray}
\label{yuanshi-a,b,c}
G(X,Y,T) &=& b^2X^3+b^{2q}Y^3+b^{2q^2}T^3+c^2X^2Y+c^{2q}Y^2T+c^{2q^2}XT^2+a^2X^2T+a^{2q}XY^2\nonumber\\
& & +~a^{2q^2}YT^2+XYT.
\end{eqnarray}
Then $G(\varepsilon,\varepsilon^q,\varepsilon^{q^2})\neq0$ for any $ \varepsilon\in\F_{q^3}^{*}$ is exactly equivalent to that $g(x)=0$ has no solutions in $\F_{q^3}^{*}$, i.e., $P_{2}(x)$ is a planar function.

If $G(X,Y,T)$ is absolutely irreducible, we know that $P_{2}(x)$ is not planar according to Lemma \ref{nonzero}. 
Hence it suffices to consider only when $G(X,Y,T)$ is not absolutely irreducible.

Thus in the following, we suppose that $\mathcal{C}_{a,b,c}$ defined by $G(X,Y,T)$ is not absolutely irreducible, which holds if and only if it contains a line $l:AX+BY+CT=0$, where $A,B,C\in\overline{\F}_{q}$.
It is readily seen from (\ref{yuanshi-a,b,c}) that $G(\varepsilon,\varepsilon^q,\varepsilon^{q^2})=G^q(\varepsilon,\varepsilon^q,\varepsilon^{q^2})$ for any $ \varepsilon\in\F_{q^3}^{*}$.
It follows that in this case the line $l^{'}:A^qY+B^q T+C^q X=0$ is also a component of $\mathcal{C}_{a,b,c}$.
	
\textbf{Case 1:} If two among $A,B$ and $C$ are zero, which is trivial.
W.l.o.g., assume $B=C=0$ and $A=1$.
At this time $G(X,Y,T)$ can only be written as $G(X,Y,T)=XYT$, which means $a=b=c=0$.
	
\textbf{Case 2:} If one among $A,B$ and $C$ is zero.
W.l.o.g., we assume $C=0$.
Therefore, we can suppose that the line $l$ has equation $X+\alpha Y=0$ for some $\alpha\in\overline{\F}_{q}^{*}$.
Obviously, the two lines $l$ and $l^{'}:Y+\alpha^q T=0$ coincide impossibly in this case.
	
Thus
\begin{equation}
\label{a,b,c-1}
G(X,Y,T)=(X+\alpha Y)(Y+\alpha^q T)(T+\alpha^{q^2} X).
\end{equation}
Expanding (\ref{a,b,c-1}), we have
\begin{eqnarray}
\label{a,b,c-expanding}
G(X,Y,T)&=&\alpha^{q^2}X^2Y+\alpha^{1+q}YT^2+\alpha Y^2T+\alpha^{q+q^2}X^2T+\alpha^q XT^2+\alpha^{1+q^2}XY^2\nonumber\\
& & +~(1+\alpha^{1+q+q^2})XYT.
\end{eqnarray}
Comparing the coefficients of (\ref{yuanshi-a,b,c}) and (\ref{a,b,c-expanding}), we find that
$$(a^2,b^2,c^2)=\Bigg(\frac{\alpha^{q+q^2}}{1+\alpha^{1+q+q^2}},0,\frac{\alpha^{q^2}}{1+\alpha^{1+q+q^2}}\Bigg),$$
where $1+\alpha^{1+q+q^2}\neq0$.
Let $\alpha=s^2$.
Then we get
$$(a,b,c)=\Bigg(\frac{s^{q+q^2}}{1+s^{1+q+q^2}},0,\frac{s^{q^2}}{1+s^{1+q+q^2}}\Bigg),$$
where $1+s^{1+q+q^2}\neq0$.
It follows from (\ref{a,b,c-1}) that for any $\varepsilon\in\F_{q^3}^{*}$, $$G(X,Y,T)=(\varepsilon+\alpha\varepsilon^q)^{1+q+q^2}.$$ 
The matrix
\begin{displaymath}
M=\left( \begin{array}{ccc}
1 & \alpha & 0 \\
0 & 1 & \alpha^q \\
\alpha^{q^2} & 0 & 1  
\end{array} \right)
\end{displaymath}
has determinant $1+\alpha^{1+q+q^2}$, which does not equal to 0. 
From Lemma \ref{Dickson}, there does not exist $\varepsilon\in\F_{q^3}^{*}$ such that $\varepsilon+\alpha\varepsilon^q=0$, that is to say, $G(\varepsilon,\varepsilon^q,\varepsilon^{q^2})\neq0$ for any $\varepsilon\in\F_{q^3}^{*}$.
Thus $P_{2}(x)$ is planar over $\F_{q^3}$.
		
\textbf{Case 3:} If none of $A,B$ and $C$ is zero.
Here we suppose that the line $l$ has equation $X+\alpha Y+\beta T=0$ for some $\alpha,\beta\in\overline{\F}_{q}^{*}$.
	
The given two lines $l$ and $l^{'}:Y+\alpha^q T+\beta^q X=0$ coincide if and only if 
\begin{displaymath}
\left\{ \begin{array}{l}
\beta=\alpha^{q+1} \\
\alpha\beta^q=1
\end{array} \right.
\iff\left\{ \begin{array}{l}
\beta=\alpha^{q+1} \\
\alpha^{q^2+q+1}=1.
\end{array} \right.
\end{displaymath}
Hence $\alpha,\beta\in\F_{q^3}^{*}$.
	
If the two lines $l$ and $l^{'}$ coincide, the determinant of the matrix
\begin{displaymath}
M=\left( \begin{array}{ccc}
1 & \alpha & \beta \\
\beta^q & 1 & \alpha^q \\
\alpha^{q^2} & \beta^{q^2} & 1  
\end{array} \right)
\end{displaymath}
vanishes.
By Lemma \ref{Dickson}, there exists some $\varepsilon\in \F_{q^3}^{*}$ such that  $\varepsilon+\alpha \varepsilon^q+\beta \varepsilon^{q^2}=0$. Then $$G(\varepsilon,\varepsilon^q,\varepsilon^{q^2}) =G_{1}(\varepsilon,\varepsilon^q,\varepsilon^{q^2})(\varepsilon+\alpha \varepsilon^q+\beta \varepsilon^{q^2})=0,$$ where $G_{1}(X,Y,T)$ is another factor of $G(X,Y,T)$ with $\deg(G_{1})=2$.
Thus $P_{2}(x)$ is not planar in this subcase.
	
If the two lines $l$ and $l^{'}$ do not coincide, then we have  
\begin{equation}
\label{a,b,c-2}
G(X,Y,T)=(X+\alpha Y+\beta T)(Y+\alpha^{q} T+\beta^{q}X)(T+\alpha^{q^2} X+\beta^{q^2}Y),
\end{equation}
where $\beta\neq\alpha^{q+1}$ or $\alpha^{q^2+q+1}\neq1$, $\alpha,\beta\in\overline{\F}_{q}^{*}$.
Expanding (\ref{a,b,c-2}), we get
\begin{eqnarray}
\label{expanding-2}
G(X,Y,T) &=& \alpha^{q^2}\beta^qX^3+\alpha\beta^{q^2}Y^3+\alpha^{q}\beta T^3\nonumber\\
& & +~(\alpha^q\beta^{1+q^2}+\beta+\alpha^{1+q})YT^2+(\alpha+\beta^{1+q^2}+\alpha^{1+q}\beta^{q^2})Y^2T\nonumber\\
& & +~ (\alpha^{q^2}\beta^{1+q}+\beta^q+\alpha^{q+q^2})X^2T+(\alpha^q+\beta^{1+q}+\alpha^{q+q^2}\beta)XT^2\nonumber\\
& & +~ (\alpha\beta^{q+q^2}+\beta^{q^2}+\alpha^{1+q^2})XY^2+(\beta^{q+q^2}+\alpha^{q^2}+\alpha^{1+q^2}\beta^q)X^2Y\nonumber\\
& & +~(1+\alpha\beta^q+\alpha^{q}\beta^{q^2}+\alpha^{q^2}\beta+\alpha^{1+q+q^2}+\beta^{1+q+q^2})XYT.
\end{eqnarray}
Comparing the coefficients of (\ref{yuanshi-a,b,c}) and (\ref{expanding-2}), we find that
$$(a^2,b^2,c^2)=\Bigg(\frac{\beta^q+\alpha^{q+q^2}+\alpha^{q^2}\beta^{1+q}}{1+\delta},\frac{\alpha^{q^2}\beta^{q}}{1+\delta},\frac{\beta^{q+q^2}+\alpha^{q^2}+\alpha^{1+q^2}\beta^{q}}{1+\delta}\Bigg),$$
where $\delta=\alpha\beta^q+\alpha^{q}\beta^{q^2}+\alpha^{q^2}\beta+\alpha^{1+q+q^2}+\beta^{1+q+q^2}\neq1$.
Meanwhile, the condition $\delta\neq1$ guarantees that $\beta=\alpha^{q+1}$ and $\alpha^{q^2+q+1}=1$ cannot hold simultaneously.
That is to say, any two factors of $G(X,Y,T)$ in (\ref{a,b,c-2}) do not coincide if $\delta\neq1$.
Furthermore, let $\alpha=u^2$ and $\beta=v^2$.
We get
$$(a,b,c)=\Bigg(\frac{v^q+u^{q+q^2}+u^{q^2}v^{1+q}}{1+\Delta},\frac{u^{q^2}v^{q}}{1+\Delta},\frac{v^{q+q^2}+u^{q^2}+u^{1+q^2}v^{q}}{1+\Delta}\Bigg),$$
where $\Delta=uv^q+u^{q}v^{q^2}+u^{q^2}v+u^{1+q+q^2}+v^{1+q+q^2}\neq1$.
	
It follows from (\ref{a,b,c-2}) that for any $\varepsilon\in\F_{q^3}^{*}$, $$G(X,Y,T)=(\varepsilon+\alpha\varepsilon^q+\beta\varepsilon^{q^2})^{1+q+q^2}.$$ 
The matrix
\begin{displaymath}
M=\left( \begin{array}{ccc}
1 & \alpha & \beta \\
\beta^q & 1 & \alpha^q \\
\alpha^{q^2} & \beta^{q^2} & 1  
\end{array} \right)
\end{displaymath}
has determinant $1+\alpha\beta^q+\alpha^{q}\beta^{q^2}+\alpha^{q^2}\beta+\alpha^{1+q+q^2}+\beta^{1+q+q^2}$, which does not equal 0. 
From Lemma \ref{Dickson}, there does not exist $\varepsilon\in\F_{q^3}^{*}$ such that $\varepsilon+\alpha\varepsilon^q+\beta\varepsilon^{q^2}=0$, that is to say, $G(\varepsilon,\varepsilon^q,\varepsilon^{q^2})\neq0$ for any $\varepsilon\in\F_{q^3}^{*}$.
Therefore, $P_{2}(x)$ is planar over $\F_{q^3}$.
	
To sum up, adding the three cases together, we finish the proof.\hfill$\square$
%\end{proof}	

\begin{rem}
\emph{In \cite{Bartoli}, the authors determined the conditions on $a,b\in\F_{q^3}$ such that $f_{a,b}(x)=ax^{q^2+1}+bx^{q+1}$ is planar over $\F_{q^3}$, which is clearly included by (2) of Theorem \ref{main} above. Moreover, our proof seems to be much simpler and the reason is that we adopt the method of L. Qu, simplifying the conditions greatly such that $P_{2}(x)$ is planar. }
\end{rem}

Next we consider another function with the type $P_{3}(x)=ax^{2(q+1)}+bx^{2(q^2+q)}+cx^{2(q^2+1)}\in\F_{q^3}[x]$ and give the proof of (3) of Theorem \ref{main}.\\
%\begin{thm}
%\label{Th-2(q^2+1)}
%Let $q=2^m$ with $m$ sufficiently large and $P_{3}(x)=ax^{2(q+1)}+bx^{2(q^2+q)}+cx^{2(q^2+1)}\in\F_{q^3}[x]$.
%Then $P_{3}(x)$ is a planar function if and only if $b=0$ and $c=a^q$.
%\end{thm}
%\begin{proof}
\textit{Proof of (3) of Theorem \ref{main}.}
According to Lemma \ref{t=3}, we know that $P_{3}(x)$ is a planar function if and only if $g(x)=0$ has no solutions in $\F_{q^3}^{*}$, where
$$
g(x)=x^{q^2+q+1}+{\rm Tr}_{\F_{q^3}/\F_{q}}\Big(x^q\Big(cx+(a\varepsilon)^{q^2}+(bx)^{q}\Big)\Big).
$$
Let 
\begin{equation}
\label{curve-2(q+1)}
G(X,Y,T)=(c+a^q)XY+(c^q+a^{q^2})YT+(c^{q^2}+a)XT+bX^2+b^qY^2+b^{q^2}T^2+XYT.
\end{equation}
Then $G(\varepsilon,\varepsilon^q,\varepsilon^{q^2})\neq0$ for any $\varepsilon\in\F_{q^3}^{*}$ is exactly equivalent to that $g(x)=0$ has no solutions in $\F_{q^3}^{*}$, i.e., $P_{3}(x)$ is a planar function.
	
If $G(X,Y,T)$ is absolutely irreducible, we know that $P_{3}(x)$ is not planar according to Lemma \ref{nonzero}. 
Hence it suffices to consider only when $G(X,Y,T)$ is not absolutely irreducible.
	
So the key work now is to determine the conditions of $a,b,c\in\F_{q^3}$ such that $\mathcal{C}_{a,b,c}$ defined by $G(X,Y,T)$ is not absolutely irreducible. 
There are two cases.
	
\textbf{Case 1:} $G(X,Y,T)=G_{1}G_{2}$, where $\deg(G_{1})=1$ and $\deg(G_{2})=2$ with $G_{2}$ being absolutely irreducible.
In this case, with the fact that  $G(\varepsilon,\varepsilon^q,\varepsilon^{q^2})=G^q(\varepsilon,\varepsilon^q,\varepsilon^{q^2})$ for any $\varepsilon\in\F_{q^3}^{*}$, we find that $G_{1}$ and $G_{1}^q$ must coincide, and the same with $G_{2}$ and $G_{2}^q$. 
Otherwise, $G(X,Y,T)=G_{1}G_{1}^qG_{1}^{q^2}G_{2}G_{2}^qG_{2}^{q^2}$, whereas the degree of $G(X,Y,T)$ is 9 other than 3, which is a contradiction.
Thus, the terms $X,Y,T$ must appear in the expression of $G_{1}$ simultaneously.
Furthermore, the three terms $X^2,Y^2,T^2$ must appear in the expression of $G_{2}$ simultaneously or not. 
The same analyses are with the terms $XY,YT,TX$ and $X,Y,T$.
	
However, it is clear that there are not the terms $X^3$ or $X^2Y$ in the expression of (\ref{curve-2(q+1)}).
Hence, $X^2,Y^2,T^2,XY,YT,TX$ cannot appear in the expression of $G_{2}$.
This indicates that the degree of $G_{2}$ cannot be 2, which contradicts to the original assumption.
Namely, this kind of decomposition of $G(X,Y,T)$ is impossible.
	
\textbf{Case 2:} $G(X,Y,T)=G_{1}G_{2}G_{3}$, where $\deg(G_{i})=1, i\in\{1,2,3\}$.
In this case, in view of that the constant term of $G(X,Y,T)$ is zero, then at least one constant term of $G_{i}$ is zero. 
W.l.o.g., let the constant term of $G_{1}$ be zero.
Assume the line $l$ corresponding to $G_{1}$ be $AX+BY+CT=0$, where $A,B,C\in\overline{\F}_{q}$.
Since  $G(\varepsilon,\varepsilon^q,\varepsilon^{q^2})=G^q(\varepsilon,\varepsilon^q,\varepsilon^{q^2})$ for any $\varepsilon\in\F_{q^3}^{*}$, then the line $l^{'}:A^qY+B^qT+C^qX=0$ is also a component of $\mathcal{C}_{a,b,c}$.
	
If the two lines $l$ and $l^{'}$ coincide.
Then none of $A,B$ and $C$ is zero.
Since if more than one among $A,B$ and $C$ are zero, the two lines coincide impossibly.
Assume now that the line $l$ is with the form $X+\alpha Y+\beta T=0$ for some $\alpha,\beta\in\overline{\F}_{q}^{*}$.
The two lines $l:X+\alpha Y+\beta T=0$ and $l^{'}:Y+\alpha^q T+\beta^q X=0$ coincide if and only if
\begin{displaymath}
\left\{ \begin{array}{l}
\beta=\alpha^{q+1} \\
\alpha\beta^q=1
\end{array} \right.
\iff\left\{ \begin{array}{l}
\beta=\alpha^{q+1} \\
\alpha^{q^2+q+1}=1.
\end{array} \right.
\end{displaymath}
Hence $\alpha,\beta\in\F_{q^3}^{*}$.
In this case, the determinant of the matrix
\begin{displaymath}
M=\left( \begin{array}{ccc}
1 & \alpha & \beta \\
\beta^q & 1 & \alpha^q \\
\alpha^{q^2} & \beta^{q^2} & 1  
\end{array} \right)
\end{displaymath}
vanishes.
By Lemma \ref{Dickson}, there exists some $\varepsilon\in \F_{q^3}^{*}$ such that  $\varepsilon+\alpha \varepsilon^q+\beta \varepsilon^{q^2}=0$. 
Then $$G(\varepsilon,\varepsilon^q,\varepsilon^{q^2}) =(\varepsilon+\alpha \varepsilon^q+\beta \varepsilon^{q^2})G_{2}(\varepsilon,\varepsilon^q,\varepsilon^{q^2})G_{3}(\varepsilon,\varepsilon^q,\varepsilon^{q^2})=0,$$ which means that $P_{3}(x)$ is not planar in this subcase.
	
If the two lines $l$ and $l^{'}$ do not coincide, then we have 
\begin{equation}
\label{ABC-1}
G(X,Y,T)=(AX+BY+CT)(A^qY+B^{q} T+C^{q}X)(A^{q^2}T+B^{q^2} X+C^{q^2}Y).
\end{equation}
It is clear that there are not the terms $X^3$ or $X^2Y$ in the expression of (\ref{curve-2(q+1)}).
Hence, two among $A,B$ and $C$ are zero.
For simplicity, set $B=C=0$ and $A=1$.
At this time, (\ref{ABC-1}) reads $G(X,Y,T)=XYT$.
Obviously, $P_{3}(x)$ is planar at this time.
	
In conclusion, we know that $P_{3}(x)$ is a planar function if and only if $G(X,Y,T)$ can only be factorized as $G(X,Y,T)=XYT$, which means that the conditions $c+a^q=0$ and $b=0$ must hold simultaneously.
Hence, the proof is finished.\hfill$\square$
%\end{proof}

\begin{rem}
\emph{Actually, (3) of Theorem \ref{main} is an extension work of \cite[Theorem 16]{Q}.
They gave the sufficient condition for $P_{3}(x)$ to be planar, and here we prove that it is also necessary indeed by the knowledge of algebraic geometry over finite fields.}
\end{rem}

\section{Planar functions over $\F_{q^4}$}
In this section, we investigate one more class of functions with the type  $P_{4}(x)=ax^{q+1}+bx^{q^2+1}+cx^{q^3+1}\in\F_{q^4}[x]$.
Then we give complete characterizations of the values of corresponding coefficients $a,b,c\in\F_{q^4}$ such that $P_{4}(x)$ is planar over $\F_{q^4}$ and give the proof of (4) of Theorem \ref{main}.\\
%\begin{thm}
%\label{F_{q^4}}
%Let $q=2^m$ with $m$ sufficiently large and $P_{4}(x)=ax^{q+1}+bx^{q^2+1}+cx^{q^3+1}\in\F_{q^4}[x]$.
%Then $P_{4}(x)$ is a planar function if and only if
%$$(a,b,c)=\Big(0,\frac{s_{1}^{q^2}}{1+s_{1}^{1+q^2}},0\Big)$$
%or $$(a,b,c)=\Bigg(\frac{s_{2}^{q+q^2+q^3}}{1+s_{2}^{1+q+q^2+q^3}},\frac{s_{2}^{q^2+q^3}}{1+s_{2}^{1+q+q^2+q^3}},\frac{s_{2}^{q^3}}{1+s_{2}^{1+q+q^2+q^3}}\Bigg),$$
%where $s_{1}\in\F_{q^4}$ such that $1+s_{1}^{1+q^2}\neq0$ and $s_{2}\in\F_{q^4}$ such that $1+s_{2}^{1+q+q^2+q^3}\neq0$.	
%\end{thm}
%\begin{proof}
\textit{Proof of (4) of Theorem \ref{main}.}
According to Lemma \ref{t=4}, we know that $P_{4}(x)$ is a planar function if and only if $g(x)=0$ has no solutions in $\F_{q^4}^{*}$, where
$$
g(x)=x^{q^3+q^2+q+1}+A_{2}^{2q+2}+(A_{3}^{2q^2+2}+A_{3}^{2q^3+2q})+(x^{q^2+1}A_{2}^{2q}+x^{q^3+q}A_{2}^2)+{\rm Tr}_{\F_{q^4}/\F_{q}}\Big(x^{q^2+q}A_{3}^2\Big),$$
with $A_{2}=bx+(bx)^{q^2}$ and $A_{3}=cx+(ax)^{q^3}$.
	
Let
\begin{eqnarray}
\label{curve-t=4}
G(X,Y,T,S) &=& XYTS+(b^{2q+2}+a^{2q}c^2)X^2Y^2+(b^{2q^3+2}+a^2c^{2q^3})X^2S^2\nonumber\\
& & +~(b^{2q+2q^2}+a^{2q^2}c^{2q})Y^2T^2+ (b^{2q^2+2q^3}+a^{2q^3}c^{2q^2})T^2S^2\nonumber\\
& & +~(c^{2q^2+2}+a^{2q^2+2})X^2T^2+(c^{2q+2q^3}+a^{2q+2q^3})Y^2S^2+b^2X^2YS\nonumber\\
& & +~ b^{2q}XY^2T+b^{2q^2}YT^2S+b^{2q^3}XTS^2+c^2X^2YT+c^{2q}Y^2TS+c^{2q^2}XT^2S\nonumber\\
& & +~c^{2q^3}XYS^2+a^{2}X^2TS+a^{2q}XY^2S+a^{2q^2}XYT^2+a^{2q^3}YTS^2.
\end{eqnarray}
Then $G(\varepsilon,\varepsilon^q,\varepsilon^{q^2},\varepsilon^{q^3})\neq0$ for any $ \varepsilon\in\F_{q^4}^{*}$ is exactly equivalent to that $g(x) = 0$ has no solutions in $\F_{q^4}^{*}$, i.e., $P_{4}(x)$ is a planar function.
	
If $G(X,Y,T,S)$ is absolutely irreducible, we know that $P_{4}(x)$ is not planar according to Lemma \ref{nonzero}. 
Hence it suffices to consider only when $G(X,Y,T,S)$ is not absolutely irreducible.
	
Now we assume that $\mathcal{C}_{a,b,c}$ defined by $G(X,Y,T,S)$ is not absolutely irreducible and determine how it decomposes.
For convenience, if $G(X,Y,T,S)$ can be factorized as a product of four linear factors, we write $G(X,Y,T,S)=(1,1,1,1)$; if $G(X,Y,T,S)$ can be factorized as a product of two quadratic absolutely irreducible factors, we write $G(X,Y,T,S)=(2,2)$.
	
\textbf{Case 1:} $G(X,Y,T,S)=(2,2)$, that is $G(X,Y,T,S)=G_{1}G_{2}$, where $\deg(G_{1})=\deg(G_{2})=2$ and both $G_{1}$ and $G_{2}$ are absolutely irreducible in this case.
Since $G(\varepsilon,\varepsilon^q,\varepsilon^{q^2},\varepsilon^{q^3})=G^q(\varepsilon,\varepsilon^q,\varepsilon^{q^2},\varepsilon^{q^3})$ for any $ \varepsilon\in\F_{q^4}^{*}$, it follows that $G_{1}^q$ is also a factor of $G(X,Y,T,S)$.
There are two subcases at this time.
	
\textbf{Subcase 1.1:} If $G_{1}$ and $G_{1}^q$ coincide, which means $G_{1}^q=\lambda G_{1}$ for some $\lambda\in\F_{q^4}^{*}$. 
Moreover, $G_{2}$ and $G_{2}^q$ also coincide.
It is readily seen that the four terms $XY,YT,TS,SX$ must appear in the expressions of $G_{1}$ and $G_{2}$ simultaneously, and the same with the terms $X^2,Y^2,T^2,S^2$ and $XT,YS$.
In view of the nonexistence of the terms $X^4$ or $X^3Y$ in the expression of (\ref{curve-t=4}), we know that $X^2,Y^2,T^2,S^2$ cannot appear in the expressions of $G_{1}$ and $G_{2}$.
	
Considering the most general subcase, we assume that $G_{1}=XY+a_{1}YT+a_{2}TS+a_{3}SX+a_{4}XT+a_{5}YS$, where $a_{i}\in\F_{q^4}^{*}$.
Then, we have $G_{1}^q=YT+a_{1}^qTS+a_{2}^qSX+a_{3}^qXY+a_{4}^qYS+a_{5}^qTX$.
The condition that $G_{1}$ and $G_{1}^q$ coincide is equivalent to the following equation:
\begin{equation*}
\label{q^4-2.2-1}
\left\{ \begin{array}{l}
a_{2}=a_{1}^{1+q} \\
a_{3}=a_{1}a_{2}^q\\
a_{1}a_{3}^q=1\\
a_{1}a_{4}^q=a_{5}\\
a_{1}a_{5}^q=a_{4}
\end{array} \right.
\iff\left\{ \begin{array}{l}
a_{1}^{1+q+q^2+q^3}=1\\
a_{2}=a_{1}^{1+q} \\
a_{3}=a_{1}^{1+q+q^2}\\
a_{1}a_{4}^q=a_{5}\\
a_{1}a_{5}^q=a_{4}.
\end{array} \right.
\end{equation*}
	
Let $\{\xi,\xi^q,\xi^{q^2},\xi^{q^3}\}$ be a normal basis of $\F_{q^4}$ over $\F_{q}$.
Consider $X=x_{0}\xi+x_{1}\xi^q+x_{2}\xi^{q^2}+x_{3}\xi^{q^3}$, $Y=x_{3}\xi+x_{0}\xi^q+x_{1}\xi^{q^2}+x_{2}\xi^{q^3}$, $T=x_{2}\xi+x_{3}\xi^q+x_{0}\xi^{q^2}+x_{1}\xi^{q^3}$ and $S=x_{1}\xi+x_{2}\xi^q+x_{3}\xi^{q^2}+x_{0}\xi^{q^3}$ with $x_{i}\in\F_{q}$, $i=\{1,2,3,4\}$.
Expanding
\begin{eqnarray*}
\psi_{1}(x_{0},x_{1},x_{2},x_{3}) &=& G_{1}(x_{0}\xi+x_{1}\xi^q+x_{2}\xi^{q^2}+x_{3}\xi^{q^3},x_{3}\xi+x_{0}\xi^q+x_{1}\xi^{q^2}+x_{2}\xi^{q^3},\nonumber\\
& &~~~~ x_{2}\xi+x_{3}\xi^q+x_{0}\xi^{q^2}+x_{1}\xi^{q^3},x_{1}\xi+x_{2}\xi^q+x_{3}\xi^{q^2}+x_{0}\xi^{q^3}),
\end{eqnarray*}
we find that the coefficient of $x_{0}^2$ in the expression of $\psi_{1}(x_{0},x_{1},x_{2},x_{3})$ is $$\xi^{1+q}+a_{1}\xi^{q+q^2}+a_{2}\xi^{q^2+q^3}+a_{3}\xi^{1+q^3}+a_{4}\xi^{1+q^2}+a_{5}\xi^{q+q^3}.$$
Furthermore, we have that
\begin{eqnarray*}
& & \xi^{1+q}+a_{1}\xi^{q+q^2}+a_{2}\xi^{q^2+q^3}+a_{3}\xi^{1+q^3}+a_{4}\xi^{1+q^2}+a_{5}\xi^{q+q^3}\nonumber\\
&=& a_{1}^{1+q+q^2+q^3}\xi^{1+q}+a_{1}\xi^{q+q^2}+a_{1}^{1+q}\xi^{q^2+q^3}+a_{1}^{1+q+q^2}\xi^{1+q^3}+a_{1}a_{5}^q\xi^{1+q^2}+a_{1}a_{4}^q\xi^{q+q^3}\nonumber\\
&=& a_{1}\Big(a_{1}^{q+q^2+q^3}\xi^{1+q}+\xi^{q+q^2}+a_{1}^{q}\xi^{q^2+q^3}+a_{1}^{q+q^2}\xi^{1+q^3}+a_{5}^q\xi^{1+q^2}+a_{4}^q\xi^{q+q^3}\Big)\nonumber\\
&=& a_{1}\Big(a_{1}^{1+q+q^2}\xi^{1+q^3}+\xi^{1+q}+a_{1}\xi^{q+q^2}+a_{1}^{1+q}\xi^{q^2+q^3}+a_{5}\xi^{q+q^3}+a_{4}\xi^{1+q^2}\Big)^q\nonumber\\
&=& a_{1}\Big(\xi^{1+q}+a_{1}\xi^{q+q^2}+a_{2}\xi^{q^2+q^3}+a_{3}\xi^{1+q^3}+a_{4}\xi^{1+q^2}+a_{5}\xi^{q+q^3}\Big)^q.
\end{eqnarray*}
It is readily seen that all the coefficients in the expression of $\psi_{1}(x_{0},x_{1},x_{2},x_{3})$ satisfy the relationship: $$c_{i}=a_{1}c_{i}^q,$$ where $c_{i}$ represents the corresponding coefficient respectively.
	
Meanwhile, the equation $t^{q-1}=a_{1}$ always has solutions in $\F_{q^4}^{*}$, since $a_{1}^{\frac{q^4-1}{\gcd(q^4-1,q-1)}}=a_{1}^{1+q+q^2+q^3}=1$.
Thus, from $c_{i}=a_{1}c_{i}^q=t^{q-1}c_{i}^q$, we get $tc_{i}=t^{q}c_{i}^q=(tc_{i})^q$.
That is, $tc_{i}\in\F_{q}$.
So the polynomial $t\psi_{1}(x_{0},x_{1},x_{2},x_{3})$ is defined over $\F_{q}$ and defines an absolutely irreducible algebraic surface. 
Then by Lemma \ref{nonzero}, there exists a nonzero element $\varepsilon\in\F_{q^4}^{*}$ such that $G_{1}(\varepsilon,\varepsilon^q,\varepsilon^{q^2},\varepsilon^{q^3})=0$, that is $G(\varepsilon,\varepsilon^q,\varepsilon^{q^2},\varepsilon^{q^3})=0$.
Thus $P_{4}(x)$ is not planar at this time.
	
\textbf{Subcase 1.2:} If $G_{1}$ and $G_{1}^q$ do not coincide. We must have $G_{1}$ and $G_{1}^{q^2}$ coincide, otherwise $G(X,Y,T,S)=G_{1}G_{1}^qG_{1}^{q^2}G_{1}^{q^3}$ and the degree of $G(X,Y,T,S)$ is 8 other than 4, which is a contradiction.
In this sense, $G(X,Y,T,S)$ can be written as $G(X,Y,T,S)=G_{1}G_{1}^q=G_{1}^{1+q}$.
Determining the solutions of $G(X,Y,T,S)=0$ is equivalent to determining the solutions of $G_{1}=0$.
	
Similar to subcase 1.1, as for $G_{1}$, let $\{\zeta,\zeta^{q^2}\}$ be a normal basis of $\F_{q^4}$ over $\F_{q^2}$.
Consider $a=a_{0}\zeta+a_{1}\zeta^{q^2}$, $b=b_{0}\zeta+b_{1}\zeta^{q^2}$, $c=c_{0}\zeta+c_{1}\zeta^{q^2}$ with $a_{i},b_{i},c_{i}\in\F_{q^2}$ and
$$\psi_{1}(x_{0},x_{1})=G_{1}(x_{0}\zeta+x_{1}\zeta^{q^2},x_{0}^q\zeta^q+x_{1}^q\zeta^{q^3},x_{0}\zeta^{q^2}+x_{1}\zeta,x_{0}^q\zeta^{q^3}+x_{1}^q\zeta^q).$$
The polynomial $\sigma\psi_{1}(x_{0},x_{1})$ is defined over $\F_{q^2}$ for some $\sigma\in\F_{q^4}^{*}$ and defines an absolutely irreducible algebraic surface.
Then by Lemma \ref{nonzero}, there exists a nonzero element $\varepsilon\in\F_{q^4}^{*}$ such that $G_{1}(\varepsilon,\varepsilon^q,\varepsilon^{q^2},\varepsilon^{q^3})=0$, that is $G(\varepsilon,\varepsilon^q,\varepsilon^{q^2},\varepsilon^{q^3})=0$.
Thus $P_{4}(x)$ is not planar at this time. 
	
\textbf{Case 2:} $G(X,Y,T,S)=(1,3)$.
Set $G_{1}$ be a factor of $G(X,Y,T,S)$ with $\deg(G_{1})=1$ in this case.
Furthermore, $G_{1}^q$ is also a factor of $G(X,Y,T,S)$, since $G(\varepsilon,\varepsilon^q,\varepsilon^{q^2},\varepsilon^{q^3})=G^q(\varepsilon,\varepsilon^q,\varepsilon^{q^2},\varepsilon^{q^3})$ for any $ \varepsilon\in\F_{q^4}^{*}$. 
We obtain that $G_{1}$ and $G_{1}^q$ coincide.
Here we assume that the surface $s$ corresponding to $G_{1}$ is $AX+BY+CT+DS=0$, where $A,B,C,D\in\overline{\F}_{q}$.
Meanwhile, the surface $s^{'}$ corresponding to $G_{1}^q$ is $A^qY+B^qT+C^qS+D^qX=0$.
It follows that if the two surfaces $s$ and $s^{'}$ coincide, then none of $A,B,C$ and $D$ is zero.
We can suppose now that $s$ is with the form $X+\alpha Y+\beta T+\gamma S=0$ for some $\alpha,\beta,\gamma\in\overline{\F}_{q}^{*}$.
Moreover, the two surfaces $s$ and $s^{'}$ coincide if and only if
\begin{equation*}
\label{kk}
\left\{ \begin{array}{l}
\beta=\alpha^{q+1} \\
\gamma=\alpha\beta^q\\
\alpha\gamma^q=1
\end{array} \right.
\iff\left\{ \begin{array}{l}
\beta=\alpha^{q+1} \\
\gamma=\alpha^{q^2+q+1}\\
\alpha^{q^3+q^2+q+1}=1.
\end{array} \right.
\end{equation*}
Hence $\alpha,\beta,\gamma\in\F_{q^4}^{*}$.
In this case, the determinant of the matrix
\begin{displaymath}
M=\left( \begin{array}{cccc}
1 & \alpha & \beta & \gamma \\
\gamma^q & 1 & \alpha^q & \beta^q \\
\beta^{q^2} & \gamma^{q^2} & 1 & \alpha^{q^2} \\  
\alpha^{q^3} & \beta^{q^3} & \gamma^{q^3} & 1
\end{array} \right)
\end{displaymath}
vanishes, since the first and second rows of the matrix are linearly dependent.
By Lemma \ref{Dickson}, there exists some $\varepsilon\in \F_{q^4}^{*}$ such that  $\varepsilon+\alpha \varepsilon^q+\beta \varepsilon^{q^2}+\gamma \varepsilon^{q^3}=0$. 
Then $$G(\varepsilon,\varepsilon^q,\varepsilon^{q^2},\varepsilon^{q^3}) =(\varepsilon+\alpha \varepsilon^q+\beta \varepsilon^{q^2}+\gamma \varepsilon^{q^3})G_{2}(\varepsilon,\varepsilon^q,\varepsilon^{q^2},\varepsilon^{q^3})=0,$$ where $G_{2}(X,Y,T,S)$ is another factor of $G(X,Y,T,S)$ with $\deg(G_{2})=3$.
Therefore, $P_{4}(x)$ is not planar at this time.
	
\textbf{Case 3:} $G(X,Y,T,S)=(1,1,2)$.
Assume that $G(X,Y,T,S)=G_{1}G_{2}G_{3}$, where $\deg(G_{1})=\deg(G_{2})=1$ and $\deg(G_{3})=2$ in this case.
As for $G_{1}$, we know that $G_{1}^q$ is also a factor of $G(X,Y,T,S)$.
Thus there are two subcases at this time.
	
\textbf{Subcase 3.1:} If $G_{1}$ and $G_{1}^q$ coincide. This subcase is similar to Case 2, and $P_{4}(x)$ is not planar at this time.
	
\textbf{Subcase 3.2:} If $G_{1}$ and $G_{1}^q$ do not coincide.
Then we have $G_{2}=G_{1}^{q}$, i.e., $G(X,Y,T,S)=G_{1}G_{1}^qG_{3}=G_{1}^{1+q}G_{3}$.
Meanwhile, $G_{1}$ and $G_{1}^{q^2}$ must coincide, since $G_{1}^{q^2}$ is also a component of $G(X,Y,T,S)$.
Here we also assume that the surface $s$ corresponding to $G_{1}$ is $AX+BY+CT+DS=0$, where $A,B,C,D\in\overline{\F}_{q}$.
In this case the surface $s^{''}$ corresponding to $G_{1}^{q^2}$ is $A^{q^2}T+B^{q^2}S+C^{q^2}X+D^{q^2}Y=0$.
It follows that if the two surfaces $s$ and $s^{''}$ coincide, the corresponding coefficients are in proportion.
In this case, the determinant of the matrix
\begin{displaymath}
M=\left( \begin{array}{cccc}
A & B & C & D \\
D^q & A^q & B^q & C^q \\
C^{q^2} & D^{q^2} & A^{q^2} & B^{q^2} \\  
B^{q^3} & C^{q^3} & D^{q^3} & A^{q^3}
\end{array} \right)
\end{displaymath}
vanishes, since the first and third rows of the matrix are linearly dependent.
By Lemma \ref{Dickson}, there exists some $\varepsilon\in \F_{q^4}^{*}$ such that  $A\varepsilon+B\varepsilon^q+C\varepsilon^{q^2}+D\varepsilon^{q^3}=0$. 
Then $$G(\varepsilon,\varepsilon^q,\varepsilon^{q^2},\varepsilon^{q^3}) =(A\varepsilon+B\varepsilon^q+C\varepsilon^{q^2}+D\varepsilon^{q^3})^{1+q}G_{3}(\varepsilon,\varepsilon^q,\varepsilon^{q^2},\varepsilon^{q^3})=0,$$
which means that $P_{4}(x)$ is not planar at this time.
	
\textbf{Case 4:} $G(X,Y,T,S)=(1,1,1,1)$.
Set $G_{1}$ be a factor of $G(X,Y,T,S)$ with $\deg(G_{1})=1$ in this case.
Assume that the surfaces $s,s^{'}$ and $s^{''}$ correspond $G_{1},G_{1}^q$ and $G_{1}^{q^2}$ respectively.
From the analysis above, we know that $P_{4}(x)$ can become planar if and only if any two of the three components $s,s^{'}$ and $s^{''}$ of $\mathcal{C}_{a,b,c}$ do not coincide.
At this time, $G(X,Y,T,S)$ can only decompose as $G(X,Y,T,S)=G_{1}G_{1}^qG_{1}^{q^2}G_{1}^{q^3}=G_{1}^{1+q+q^2+q^3}$.
Assume that the surface $s$ corresponding to $G_{1}$ is $AX+BY+CT+DS=0$, where $A,B,C,D\in\overline{\F}_{q}$.
In view of the nonexistence of the terms $X^4$ or $X^3Y$ in the expression of (\ref{curve-t=4}), it follows that two among $A,B,C$ and $D$ are zero.
Hence, there only exist two possible factorizations of $G(X,Y,T,S)$.
	
\textbf{Subcase 4.1:} 
\begin{equation}
\label{factorization-1}
G(X,Y,T,S)=(X+\theta_{1}T)(Y+\theta_{1}^qS)(T+\theta_{1}^{q^2}X)(S+\theta_{1}^{q^3}Y),
\end{equation}
where $\theta_{1}^{q^2+1}\neq1$.
Otherwise, the two surfaces $X+\theta_{1}T=0$ and $T+\theta_{1}^{q^2}X=0$ or the other two surfaces $Y+\theta_{1}^qS=0$ and $S+\theta_{1}^{q^3}Y=0$ would coincide indeed, which indicates that $P_{4}(x)$ is not planar.
Expanding (\ref{factorization-1}), we have
\begin{eqnarray}
\label{ex-1}
G(X,Y,T,S) &=& (1+\theta_{1}^{1+q^2})(1+\theta_{1}^{q+q^3})XYTS+\theta_{1}^{q^2+q^3}X^2Y^2+\theta_{1}^{1+q^3}Y^2T^2+\theta_{1}^{1+q}T^2S^2\nonumber\\
& & +~\theta_{1}^{q+q^2}X^2S^2+ \theta_{1}^{q^3}(1+\theta_{1}^{1+q^2})XY^2T+\theta_{1}(1+\theta_{1}^{q+q^3})YT^2S\nonumber\\
& & +~\theta_{1}^{q}(1+\theta_{1}^{1+q^2})XTS^2 +\theta_{1}^{q^2}(1+\theta_{1}^{q+q^3})X^2YS.
\end{eqnarray}
Comparing the coefficients of (\ref{curve-t=4}) and (\ref{ex-1}), we find that
$$(a^2,b^2,c^2)=\Bigg(0,\frac{\theta_{1}^{q^2}}{1+\theta_{1}^{1+q^2}},0\Bigg),$$
where $\theta_{1}^{q^2+1}\neq1$.
Let $\theta_{1}=s_{1}^2$.
We have
$$(a,b,c)=\Bigg(0,\frac{s_{1}^{q^2}}{1+s_{1}^{1+q^2}},0\Bigg),$$
where $s_{1}^{q^2+1}\neq1$.
	
It follows from (\ref{factorization-1}) that for any $\varepsilon\in\F_{q^4}^{*}$, $$G(X,Y,T,S)=(\varepsilon+\theta_{1}\varepsilon^{q^2})^{1+q+q^2+q^3}.$$ 
The matrix
\begin{displaymath}
M=\left( \begin{array}{cccc}
1 & 0 & \theta_{1} & 0 \\
0 & 1 & 0 & \theta_{1}^q \\
\theta_{1}^{q^2} & 0 & 1 & 0 \\  
0 & \theta_{1}^{q^3} & 0 & 1
\end{array} \right)
\end{displaymath}
has determinant $(1+\theta_{1}^{1+q^2})(1+\theta_{1}^{q+q^3})=(1+\theta_{1}^{1+q^2})^{q+1}$, which does not equal 0. 
From Lemma \ref{Dickson}, there does not exist $\varepsilon\in\F_{q^4}^{*}$ such that $\varepsilon+\theta_{1}\varepsilon^{q^2}=0$, that is to say, $G(\varepsilon,\varepsilon^q,\varepsilon^{q^2},\varepsilon^{q^3})\neq0$ for any $\varepsilon\in\F_{q^4}^{*}$.
Hence, $P_{4}(x)$ is planar over $\F_{q^4}$.
	
\textbf{Subcase 4.2:} 
\begin{equation}
\label{factorization-2}
G(X,Y,T,S)=(X+\theta_{2}Y)(Y+\theta_{2}^qT)(T+\theta_{2}^{q^2}S)(S+\theta_{2}^{q^3}X).
\end{equation}
Expanding (\ref{factorization-2}), we have
\begin{eqnarray}
\label{ex-2}
G(X,Y,T,S) &=& (1+\theta_{2}^{1+q^2+q^2+q^3})XYTS+\theta_{2}Y^2TS+\theta_{2}^{q}XT^2S+\theta_{2}^{q^2}XYS^2+\theta_{2}^{q^3}X^2YT\nonumber\\
& & +~ \theta_{2}^{1+q}YT^2S+\theta_{2}^{q+q^2}XTS^2+\theta_{2}^{q^2+q^3}X^2YS+\theta_{2}^{1+q^3}XY^2T
\nonumber\\
& & +~ \theta_{2}^{1+q+q^2}YTS^2+\theta_{2}^{1+q+q^3}XYT^2+\theta_{2}^{1+q^2+q^3}XY^2S+\theta_{2}^{q+q^2+q^3}X^2TS\nonumber\\
& & +~\theta_{2}^{1+q^2}Y^2S^2+\theta_{2}^{q+q^3}X^2T^2.
\end{eqnarray}
Comparing the coefficients of (\ref{curve-t=4}) and (\ref{ex-2}), we find that
$$(a^2,b^2,c^2)=\Bigg(\frac{\theta_{2}^{q+q^2+q^3}}{1+\theta_{2}^{1+q+q^2+q^3}},\frac{\theta_{2}^{q^2+q^3}}{1+\theta_{2}^{1+q+q^2+q^3}},\frac{\theta_{2}^{q^3}}{1+\theta_{2}^{1+q+q^2+q^3}}\Bigg),$$
where $\theta_{2}^{1+q+q^2+q^3}\neq1$.
Let $\theta_{2}=s_{2}^2$.
We have
$$(a,b,c)=\Bigg(\frac{s_{2}^{q+q^2+q^3}}{1+s_{2}^{1+q+q^2+q^3}},\frac{s_{2}^{q^2+q^3}}{1+s_{2}^{1+q+q^2+q^3}},\frac{s_{2}^{q^3}}{1+s_{2}^{1+q+q^2+q^3}}\Bigg),$$
where $s_{2}^{1+q+q^2+q^3}\neq1$.
	
It follows from (\ref{factorization-2}) that for any $\varepsilon\in\F_{q^4}^{*}$, $$G(X,Y,T,S)=(\varepsilon+\theta_{2}\varepsilon^{q})^{1+q+q^2+q^3}.$$ 
The matrix
\begin{displaymath}
M=\left( \begin{array}{cccc}
1 & \theta_{2} & 0 & 0 \\
0 & 1 & \theta_{2}^q & 0 \\
0 & 0 & 1 & \theta_{2}^{q^2} \\  
\theta_{2}^{q^3} & 0 & 0 & 1
\end{array} \right)
\end{displaymath}
has determinant $1+\theta_{2}^{1+q^2+q^2+q^3}$, which does not equal 0. 
From Lemma \ref{Dickson}, there does not exist $\varepsilon\in\F_{q^4}^{*}$ such that $\varepsilon+\theta_{2}\varepsilon^{q}=0$, that is to say, $G(\varepsilon,\varepsilon^q,\varepsilon^{q^2},\varepsilon^{q^3})\neq0$ for any $\varepsilon\in\F_{q^4}^{*}$.
So $P_{4}(x)$ is planar over $\F_{q^4}$ indeed.
	
Thus the whole proof is completed.\hfill$\square$
%\end{proof}

\begin{rem}
\emph{The former case of (4) of Theorem \ref{main} is exactly the case of monomials $ax^{q+1}$ in $\F_{q^2}$, see Proposition \ref{M=N} for details.
In \cite{Q}, L. Qu considered two explicit classes of constructions of planar functions over $\F_{q^4}$ and gave sufficient conditions for them to be planar.
However, for the latter case of (4) of Theorem \ref{main}, it is the first time that we investigate the necessary and sufficient conditions for this class of functions to be planar.}
\end{rem}

\section{The equivalence between the derived semifields and corresponding fields}

In this section, we mainly focus on the equivalence between the semifields produced by the planar functions in Theorem \ref{main} and the corresponding fields.

Firstly, the planar functions in $\F_{q^2}$ cannot be new, since they are monomials exactly and have already been investigated by K.-U. Schmidt and Y. Zhou \cite{Schmidt} in 2014.

Next, the planar functions in $\F_{q^3}$ cannot be new either.
The reason is that they are all of Dembowski-Ostrom type, which means that the semifields’ centers must contain $\F_{q}$.
By the classification of semifields of order $q^3$ over $\F_{q}$ by Menichetti \cite{Menichetti} in 1977, they must be finite fields. Therefore, these functions should be equivalent to $F(x)=0$.

Finally, we discuss the planar functions over $\F_{q^4}$, which are $P_{4}(x)=ax^{q+1}+bx^{q^2+1}+cx^{q^3+1}\in\F_{q^4}[x]$ in (4) of Theorem \ref{main}.	
We have  tried to analyze in theory and done many experiments by Mathematica and Magma. Unfortunately, we conclude that these semifields are exactly fields to a large extent.

(1) In theory:

As for the proof in theory, the main difficulty is to construct the semifield $(\F_{q^4},+,\circ)$ and especially the multiplication $\circ$.
Towards this object, we need to calculate the  inverse $L^{-1}(x)$ of some linearized polynomial deduced by $L(x)=x*1$, where $*$ is the multiplication of the presemifield $(\F_{q^4},+,*)$ corresponding to $P_{4}(x)$ such that $x*y=xy+P_{4}(x+y)+P_{4}(x)+P_{4}(y)$.
Then we obtain the semifield multiplication $\circ$, that is $x\circ y=L^{-1}(x*y)$.
However, in general case the inverse $L^{-1}(x)$ has too many terms and is very complicated. Let $L^{-1}(x)=a_{1}x+a_{2}x^q+a_{3}x^{q^2}+a_{4}x^{q^3}$. 
Then  computers hint that each $a_i$ can be expressed as a fractional function of $s_{2}$, while its denominator has more than 50 terms.  Hence it is too complicated to go on proving the equivalence of the semifield and the field theoretically.
We are also wondering whether there are some other methods or ideas to deal with it.

Thus, we consider a concrete example that $(a,b,c)=(\omega,1,\omega^2)$, where $\omega^2+\omega+1=0$ and $m$ is even.
Through complex and tedious computation of the left nucleus of the corresponding semifield, we prove that it equals to the field $\F_{q^4}$ indeed.

(2) In experiments:

We test all the elements $(a,b,c)$ such that $P_4$'s are planar for the case $m=2$.
Unfortunately, we find that the semifields produced by the planar functions are exactly fields.
For $m=3$, we just test several $(a,b,c)$ due to the limit of the computational power, and these semifields are also equivalent to fields.
For $m>3$, it seems to be too hard and complex to verify the equivalence of the corresponding semifield and the field.

All in all, we spend much time on both theory and experiments. 
We prove for a concrete example that the corresponding semifield is field. 
The experiments also hint negative results. 
However, whether the semifields corresponding to the planar functions are fields or not is still unclear, which is left as an open problem.

In the following, we give the concrete example for a better understanding.
\begin{ex}
Let $H(x)=\omega x^{q+1}+x^{q^2+1}+\omega^2x^{q^3+1}\in\F_{q^4}[x]$, where $q=2^m$, $\omega^2+\omega+1=0$ and $m$ is even. 
Then the semifield derived from $H$ is isomorphic to the finite field.
\end{ex}
\begin{proof}
Let us define the following multiplication $*$ that 
\begin{eqnarray*}
x*y &=& xy+H(x+y)+H(x)+H(y)\nonumber\\
&=& xy+\omega(x^qy+xy^q)+(x^{q^2}y+xy^{q^2})+\omega^2(x^{q^3}y+xy^{q^3}).
\end{eqnarray*}
Since $x*1=x+\omega x^q+x^{q^2}+\omega^2x^{q^3}$, $(\F_{q^4},+,*)$ is not a semifield but a presemifield.
Denote $L(x)=x*1=x+\omega x^q+x^{q^2}+\omega^2x^{q^3}$.
According to the method introduced by B. Wu in 2014 (see \cite[Theorem 1.1]{Wu} for detail) and with the fact that $\omega^q=\omega$, we obtain  $$L^{-1}(x)=x+\omega^2x^q+x^{q^2}+\omega x^{q^3}.$$
Then we define
\begin{eqnarray*}
x\circ y &=& L^{-1}(x*y)\nonumber\\
&=& xy+x^q(\omega^2y^q+\omega y^{q^2}+y^{q^3})+x^{q^2}(\omega y^q+y^{q^2}+\omega^2y^{q^3})+x^{q^3}(y^q+\omega^2y^{q^2}+\omega y^{q^3}).
\end{eqnarray*}
Hence $(\F_{q^4},+,\circ)$ is a semifield corresponding to $H$.
Furthermore, to check whether it is new or not, we determine the left nucleus of the derived semifields.

On one hand, we have
$$\alpha\circ(x\circ y)=\alpha A_{0}(x,y)+\alpha^q A_{1}(x,y)+\alpha^{q^2} A_{2}(x,y)+\alpha^{q^3} A_{3}(x,y),$$
where
\begin{eqnarray*}
A_{0}(x,y) & = & x\circ y, \\
A_{1}(x,y) & = & \omega^2(x\circ y)^q+\omega(x\circ y)^{q^2}+(x\circ y)^{q^3}, \\
A_{2}(x,y) & = & \omega(x\circ y)^q+(x\circ y)^{q^2}+\omega^2(x\circ y)^{q^3}, \\
A_{3}(x,y) & = & (x\circ y)^q+\omega^2(x\circ y)^{q^2}+\omega(x\circ y)^{q^3}.
\end{eqnarray*}

On the other hand, we have
\begin{eqnarray*}
(\alpha\circ x)\circ y &=& (\alpha\circ x)y+(\alpha\circ x)^q(\omega^2y^q+\omega y^{q^2}+y^{q^3})+(\alpha\circ x)^{q^2}(\omega y^q+y^{q^2}+\omega^2y^{q^3})\nonumber\\
& & +~(\alpha\circ x)^{q^3}(y^q+\omega^2y^{q^2}+\omega y^{q^3})\nonumber\\
&=& \alpha B_{0}(x,y)+\alpha^q B_{1}(x,y)+\alpha^{q^2} B_{2}(x,y)+\alpha^{q^3} B_{3}(x,y),
\end{eqnarray*}
where
\begin{eqnarray*}
B_{0}(x,y) & = & xy+(x^{q^2}+\omega^2x^{q^3}+\omega x)(\omega^2y^q+\omega y^{q^2}+y^{q^3})\\
& & +~(\omega x^{q^3}+x+\omega^2 x^q)(\omega y^q+y^{q^2}+\omega^2y^{q^3}) \\
& & +~(\omega^2x+\omega x^q+x^{q^2})(y^q+\omega^2y^{q^2}+\omega y^{q^3}), \\
B_{1}(x,y) & = & (\omega^2x^q+\omega x^{q^2}+x^{q ^3})y+x^q(\omega^2y^q+\omega y^{q^2}+y^{q^3})\\
& & +~(x^{q^3}+\omega^2x+\omega x^q)(\omega y^q+y^{q^2}+\omega^2y^{q^3}) \\
& & +~(\omega x+x^q+\omega^2x^{q^2})(y^q+\omega^2y^{q^2}+\omega y^{q^3}), \\
B_{2}(x,y) & = & (\omega x^{q}+x^{q^2}+\omega^2x^{q^3})y+(\omega^2x^{q^2}+\omega x^{q^3}+x)(\omega^2y^q+\omega y^{q^2}+y^{q^3})\\
& & +~x^{q^2}(\omega y^q+y^{q^2}+\omega^2y^{q^3}) \\
& & +~(x+\omega^2x^q+\omega x^{q^2})(y^q+\omega^2y^{q^2}+\omega y^{q^3}), \\
B_{3}(x,y) & = & (x^q+\omega^2x^{q^2}+\omega x^{q^3})y+(\omega x^{q^2}+x^{q^3}+\omega^2x)(\omega^2y^q+\omega y^{q^2}+y^{q^3})\\
& & +~(\omega^2x^{q^3}+\omega x+x^q)(\omega y^q+y^{q^2}+\omega^2y^{q^3}) \\
& & +~x^{q^3}(y^q+\omega^2y^{q^2}+\omega y^{q^3}).
\end{eqnarray*}
Then a direct computation shows that
$$A_{i}(x,y)=B_{i}(x,y), i=0,1,2,3.$$
Hence $$\alpha\circ(x\circ y)=(\alpha\circ x)\circ y\textup{~~for all~} \alpha,x,y\in\F_{q^4},$$
which means that $(\F_{q^4},+,\circ)$ is isomorphic to the finite field $\F_{q^4}$.
\end{proof}

%Finally, we would like to point out that, for this manuscript, the main contribution was the complete characterizations of four classes of planar functions over $\F_{q^k}$ by using algebraic geometry over finite fields and
%the theory of curve (hypersurface), where $q=2^m$ and $k=2,3,4$ respectively. 

\section{Conclusion}
The discoveries based on the basic tools from algebraic geometry over finite fields and the theory of curve (hypersurface) are pretty interesting and challenging. In this paper, we investigate several types of planar functions by combining the methods of L. Qu \cite{Q} and the Lang-Weil bound. 
As far as we know, it is the first time to combine these two methods to study planar functions. There are mainly two achievements. For one thing, we completely provide the asymptotic conditions for four classes of polynomials over $\F_{q^2}$, $\F_{q^3}$ and $\F_{q^4}$ to be planar (i.e., Theorem \ref{main}). To our knowledge, before this paper, only one class of polynomials with necessary and sufficient conditions to be planar was characterized by D. Bartoli and M. Timpanella \cite{Bartoli}.
Furthermore, we generalize their results greatly.
For the other thing, by using the Lang-Weil bound \cite{ACGM,SLAW} and the approach of L. Qu \cite{Q}, our proof seems to be much simpler compared to that of D. Bartoli and M. Timpanella \cite{Bartoli}, and the reason is that the curve $\mathcal{C}$ what we investigate becomes much simpler.
Some of our results are exactly generalizations of known ones.
We hope that these results will supply more choices and methods to study planar functions, projective planes and other applications in mathematics.

%\bigskip 
%{\ \ \ \ \ \ \ \ \ \ \ \ \ \ \ \ \ \ \ \ \ \ \ \ \ \ \ \ \ \ \ \ \ \ \ \ \ \ \ \ \ \ \ \ \ \ \ \ REFERENCES}

\end{document}